\documentclass[12pt]{amsart}

\usepackage{amssymb,amsmath,graphicx,color,textcomp, amsthm,bbm,bbold, enumerate,booktabs}
\usepackage{chngcntr}
\usepackage{apptools}
\usepackage{mathrsfs}
\usepackage{tikz}
\usetikzlibrary{trees}

\AtAppendix{\counterwithin{theorem}{section}}
\definecolor{Red}{cmyk}{0,1,1,0}

\definecolor{verde}{cmyk}{1,0,1,0}

\definecolor{loka}{cmyk}{.5,0,1,.5}
\definecolor{azul}{cmyk}{1,1,0,0}

\numberwithin{equation}{section}

\newcommand{\be}{\begin{equation}}
\newcommand{\ee}{\end{equation}}

\newtheorem{theorem}{Theorem}

\newtheorem{definition}{Definition}
\newtheorem{lemma}{Lemma}
\newtheorem{remark}{Remark}

\begin{document}
\title{On the $\psi$-Hilfer fractional derivative}
\author{J. Vanterler da C. Sousa$^1$}
\address{$^1$ Department of Applied Mathematics, Institute of Mathematics,
 Statistics and Scientific Computation, University of Campinas --
UNICAMP, rua S\'ergio Buarque de Holanda 651,
13083--859, Campinas SP, Brazil\newline
e-mail: {\itshape \texttt{ra160908@ime.unicamp.br, capelas@ime.unicamp.br }}}

\author{E. Capelas de Oliveira$^1$}

\begin{abstract}In this paper we introduce a new fractional derivative with respect to another function the so-called $\psi$-Hilfer fractional derivative. We discuss some properties and important results of the fractional calculus. In this sense, we present some uniformly convergent sequence of function results and examples involving the Mittag-Leffler function with one parameter. Finally, we present a wide class of integrals and fractional derivatives, by means of the fractional integral with respect to another function and the $\psi$-Hilfer fractional derivative.

\vskip.5cm
\noindent
\emph{Keywords}: Fractional calculus, $\psi$-Hilfer fractional derivative, a class of fractional derivatives and integrals.
\newline 
MSC 2010 subject classifications. 26A03; 26A15; 26A33; 26A46 .
\end{abstract}
\maketitle

\section{Introduction}
Fractional calculus has caught the attention of many researchers over the last few decades as it is a solid and growing work both in theory and in its applications \cite{AHMJ,SAM}. The importance of fractional calculus growth is notable not only in pure and applied mathematics but also in physics, chemistry, engineering, biology, and other \cite{ROSA,FEE,ATA,COR,MAG,JEM,MAGIN,MERA,FJC}.

Since the beginning of the fractional calculus in 1695 \cite{GWL1,GWL2,GWL3}, there are numerous definitions of integrals and fractional derivatives, and over time, new derivatives and fractional integrals arise. These integrals and fractional derivatives have a different kernel and this makes the number of definitions \cite{AHMJ,SAM,RHM,UNT,UNT1,ECJT,AGRA} wide.

With the wide number of definitions of integrals and fractional derivatives, it was necessary to introduce a fractional derivative of a function $f$ with respect to another function, making use of the fractional derivative in the Riemann-Liouville sense, given by \cite{AHMJ}
\begin{eqnarray*}
\mathcal{D}_{a+}^{\alpha ;\psi }f\left( x\right) &=&\left( \frac{1}{\psi ^{\prime
}\left( x\right) }\frac{d}{dx}\right) ^{n}I_{a+}^{n-\alpha ;\psi }f\left(
x\right),
\end{eqnarray*}
where $n-1<\alpha<n$, $n=[\alpha]+1$ for $\alpha\notin \mathbb{N}$ and $n=\alpha$ for $\alpha\in\mathbb{N}$.
However, such a definition only encompasses the possible fractional derivatives that contain the differentiation operator acting on the integral operator.

In the same way, recently, Almeida \cite{RCA} using the idea of the fractional derivative in the Caputo sense, proposes a new fractional derivative called $\psi$-Caputo derivative with respect to another function $\psi$, which generalizes a class of fractional derivatives, whose definition is given by
\begin{equation*}
^{C}D_{a+}^{\alpha ;\psi }f\left( x\right) =I_{a+}^{n-\alpha ;\psi }\left( 
\frac{1}{\psi ^{\prime }\left( x\right) }\frac{d}{dx}\right) ^{n}f\left(
x\right)
\end{equation*}
where $n-1<\alpha<n$, $n=[\alpha]+1$ for $\alpha\notin \mathbb{N}$ and $n=\alpha$ for $\alpha\in\mathbb{N}$.

Although such definitions are very general, there exist the possibility of proposing a fractional differentiable operator that unifies these above operators and can overcome the wide number of definitions. In this perspective, we will use the Hilfer fractional derivative idea \cite{HILFER}, and propose a fractional differential operator of a function with respect to another $\psi$ function, the so-called $\psi$-Hilfer derivative. The advantage of the fractional operator proposed here, is the freedom of choice of the classical differential operator, that is, once it acts on the fractional integral operator, once the fractional integral operator acts on the differential operator. Thus, the class of fractional derivatives derived from the $\psi$-Hilfer operator is in fact larger, making the fractional operator a generalization of the fractional operators above defined.

The paper is organized as follows. In Section 2 we begin with the definition of some spaces of functions, among them, the weighted spaces. In this section, we present the definition of fractional integral of a function $f$ with respect to another function $\psi$ and the definitions of fractional derivatives $\psi$-Riemann-Liouville and $\psi$-Caputo. From these definitions, some important results were introduced for the development of the paper. In section 3, we present our main result, the definition of $\psi$-Hilfer fractional derivative. We discuss some properties of the fractional operator: the identity, is limited, the relation with the operators $\psi$-Riemann-Liouville and $\psi$-Caputo as well as the relation with the fractional integral operator. In Section 4, we present some results involving uniformly convergent sequence of functions and some examples involving Mittag-Leffler functions and the $\psi$ function. Section 5 deals with the class of integrals and fractional derivatives derived from the fractional $\psi$-Hilfer operator and the integral fractional operator. Concluding remarks close the paper.

%%%%%%%%%%%%%%%%%%%%%%%%%%%%%%%%%%%%%%%%%%%%%%%%%%%%%%%%%%%%%%%%%%%%%%%%%%%%%%%%%%%%%%%%%%%%%%%%%%%%%%%%%%%%%%%%%%%%%%%%%%%%%%%%%%%%%%%%%%%%%%%%%%%%%%%%%%%%
\section{Preliminaries}

In this section, we present weighted spaces and some new concepts related to the fractional integrals and derivatives of a function $f$ with respect to another function $\psi$. In this sense, some results  that will be important in the course of the paper, will be mentioned.

Let $[a,b]$ $(0<a<b<\infty)$ be a finite interval on the half-axis $\mathbb{R}^{+}$ and $C[a,b]$, $AC^{n}[a,b]$, $C^{n}[a,b]$ be the spaces of continuous functions, $n$-times absolutely continuous, $n$-times continuous and continuously differentiable functions on $[a,b]$, respectively. 

The space of the continuous function $f$ on $[a,b]$ with the norm is defined by \cite{AHMJ}
\begin{equation*}
\left\Vert f\right\Vert _{C\left[ a,b\right] }=\underset{t\in \left[ a,b \right] }{\max }\left\vert f\left( t\right) \right\vert.
\end{equation*}

On the order hand, we have $n$-times absolutely continuous given by
\begin{equation*}
AC^{n}\left[ a,b\right] =\left\{ f:\left[ a,b\right] \rightarrow \mathbb{R}
;\text{ }f^{\left( n-1\right) }\in AC\left( \left[ a,b\right] \right)
\right\} .
\end{equation*}

The weighted space $C_{\gamma,\psi}[a,b]$ of functions $f$ on $[a,b]$ is defined by
\begin{equation*}
C_{\gamma ;\psi }\left[ a,b\right] =\left\{ f:\left( a,b\right] \rightarrow 
\mathbb{R};\text{ }\left( \psi \left( t\right) -\psi \left( a\right) \right)
^{\gamma }f\left( t\right) \in C\left[ a,b\right] \right\} ,\text{ }0\leq \gamma <1
\end{equation*}
with the norm
\begin{equation*}
\left\Vert f\right\Vert _{C_{\gamma ;\psi }\left[ a,b\right] }=\left\Vert
\left( \psi \left( t\right) -\psi \left( a\right) \right) ^{\gamma}f\left(
t\right) \right\Vert _{C\left[ a,b\right] }=\underset{t\in \left[ a,b\right] 
}{\max }\left\vert \left( \psi \left( t\right) -\psi \left( a\right) \right)
^{\gamma }f\left( t\right) \right\vert.
\end{equation*}

The weighted space $C_{\delta ;\psi }^{n}\left[ a,b\right]$ of function $f$ on $[a,b]$ is defined by
\begin{equation*}
C_{\gamma;\psi }^{n}\left[ a,b\right] =\left\{ f:\left( a,b\right]
\rightarrow \mathbb{R};\text{ }f\left( t\right) \in C^{n-1}\left[ a,b\right] ;\text{ }f^{\left( n\right) }\left( t\right) \in C_{\gamma;\psi }\left[ a,b\right] \right\} ,\text{ }0\leq \gamma <1
\end{equation*}
with the norm
\begin{equation*}
\left\Vert f\right\Vert _{C_{\gamma ;\psi }^{n}\left[ a,b\right] }=\overset{n-1}{\underset{k=0}{\sum }}\left\Vert f^{\left( k\right) }\right\Vert _{C\left[ a,b\right] }+\left\Vert f^{\left( n\right) }\right\Vert _{C_{\gamma ;\psi }\left[ a,b\right] }.
\end{equation*}

For $n=0$, we have, $C_{\gamma }^{0}\left[ a,b\right] =C_{\gamma }\left[ a,b\right] $.

The weighted space $C^{\alpha,\beta}_{\gamma,\psi}[a,b]$ is defined by
\begin{equation*}
C_{\gamma ;\psi }^{\alpha ,\beta }\left[ a,b\right] =\left\{ f\in C_{\gamma;\psi }\left[ a,b\right] ;\text{ }^{H}\mathbb{D}_{a+}^{\alpha ,\beta ;\psi }f\in C_{\gamma;\psi }\left[ a,b\right] \right\} ,\text{ }\gamma =\alpha +\beta \left( 1-\alpha\right) .
\end{equation*}

\begin{definition}{\rm \cite{AHMJ}} Let $[a,b]$ $(-\infty<a<b<\infty)$ be a finite interval on the real-axis $\mathbb{R}$. The Riemann-Liouville fractional integrals $(\mbox{left-sided and right-sided})$ of order $\alpha$, with $\alpha>0$, are defined by
\begin{equation}\label{I1}
^{RL}I_{a+}^{\alpha }f\left( x\right) :=\frac{1}{\Gamma \left( \alpha
\right) }\int_{a}^{x}\frac{f\left( t\right) }{\left( x-t\right) ^{1-\alpha }}dt, \text{ }x>a
\end{equation}
and
\begin{equation*}
^{RL}I_{b-}^{\alpha }f\left( x\right) :=\frac{1}{\Gamma \left( \alpha \right) }\int_{x}^{b}\frac{f\left( t\right) }{\left( t-x\right) ^{1-\alpha }}dt, \text{ }x<b,
\end{equation*}
respectively.
\end{definition}

\begin{definition}{\rm \cite{AHMJ}} Let $I=(a,b)$ and $f(x)\in AC^{n}(a,b)$ and $n-1<\alpha <n$, $n\in\mathbb{N}_{0}$. The Riemann-Liouville fractional derivatives $(\mbox{left-sided and right-sided})$ of function $f$ of order $\alpha$, are given by
\begin{eqnarray}\label{D1}
\mathcal{D}_{a+}^{\alpha }f\left( x\right)  &=&\left( \frac{d}{dx}\right)
^{n}I_{a+}^{n-\alpha }f\left( x\right)   \notag \\
&=&\frac{1}{\Gamma \left( n-\alpha \right) }\left( \frac{d}{dx}\right)
^{n}\int_{a}^{x}\left( x-t\right) ^{n-\alpha -1}f\left( t\right) dt
\end{eqnarray}
and
\begin{eqnarray*}
\mathcal{D}_{b-}^{\alpha }f\left( x\right)  &=&\left( -1\right) ^{n}\left( \frac{d}{dx}\right) ^{n}I_{b-}^{n-\alpha }f\left( x\right)   \notag \\
&=&\frac{\left( -1\right) ^{n}}{\Gamma \left( n-\alpha \right) }\left( \frac{d}{dx}\right) ^{n}\int_{x}^{b}\left( t-x\right) ^{n-\alpha -1}f\left(t\right) dt,
\end{eqnarray*}
respectively.
\end{definition}

\begin{definition}{\rm \cite{SAM,SRI}} The Hilfer fractional derivatives $(\mbox{left-sided and right-sided})$ $D^{\alpha,\beta}_{a+}$ of function $f\in C^{n}(a,b)$ of order $n-1<\alpha<n$ and type $0\leq \beta \leq 1$, are defined by
\begin{equation*}
D_{a+}^{\alpha ,\beta }f\left( x\right) =I_{a+}^{\gamma -\alpha }\left( 
\frac{d}{dx}\right) ^{n}I_{a+}^{\left( 1-\beta \right) \left( n-\alpha
\right) }f\left( x\right) 
\end{equation*}
and
\begin{equation*}
D_{b-}^{\alpha ,\beta }f\left( x\right) =I_{b-}^{\gamma -\alpha }\left( -\frac{d}{dx}\right) ^{n}I_{b-}^{\left( 1-\beta \right) \left( n-\alpha
\right) }f\left( x\right) 
\end{equation*}
where $I^{\alpha}_{a+}$ and $\mathcal{D}^{\alpha}_{a+}$ are Riemann-Liouville fractional integral and derivative given by {\rm Eq.(\ref{I1})} and {\rm Eq.(\ref{D1})}, respectively.
\end{definition}

In addition to the above definition of fractional integral, we will see in section 5 a class of integrals and fractional derivatives. Due to the huge amount of definitions, i.e., fractional operators, the following definition is a special approach when the kernel is unknown, involving a function $\psi $.

\begin{definition}{\rm \cite{AHMJ}} Let $(a,b)$ $(-\infty \leq a<b \leq \infty)$ be a finite or infinite interval of the real line $\mathbb{R}$ and $\alpha>0$. Also let $\psi(x)$ be an increasing and positive monotone function on $(a,b]$, having a continuous derivative $\psi'(x)$ on $(a,b)$. The left and right-sided fractional integrals of a function $f$ with respect to another function $\psi$ on $[a,b]$ are defined by
\begin{equation}\label{A}
I_{a+}^{\alpha ;\psi }f\left( x\right) =\frac{1}{\Gamma \left( \alpha
\right) }\int_{a}^{x}\psi ^{\prime }\left( t\right) \left( \psi \left(
x\right) -\psi \left( t\right) \right) ^{\alpha -1}f\left( t\right) dt
\end{equation}
and
\begin{equation}\label{I3}
I_{b-}^{\alpha ;\psi }f\left( x\right) =\frac{1}{\Gamma \left( \alpha
\right) }\int_{x}^{b}\psi ^{\prime }\left( t\right) \left( \psi \left(
t\right) -\psi \left( x\right) \right) ^{\alpha -1}f\left( t\right) dt.
\end{equation}
\end{definition}

We will present the lemmas and theorem below, for details of their respective proves, see {\rm \cite{AHMJ,RCA}}.

\begin{lemma} \label{LE} Let $\alpha>0$ and $\beta>0$. Then, we have the following semigroup property given by
\begin{equation*}
I_{a+}^{\alpha ;\psi }I_{a+}^{\beta ;\psi }f\left( x\right) =I_{a+}^{\alpha
+\beta ;\psi }f\left( x\right)
\end{equation*}
and
\begin{equation*}
I_{b-}^{\alpha ;\psi }I_{b-}^{\beta ;\psi }f\left( x\right) =I_{b-}^{\alpha
+\beta ;\psi }f\left( x\right).
\end{equation*}
\end{lemma}

\begin{proof}
See {\rm \cite{AHMJ}}.
\end{proof}

%%%%%%%%%%%%%%%%%%%%%%%%%%%
\begin{lemma} \label{LE1} Let $\alpha>0$ and $\delta>0$.
\begin{enumerate}
\item If $f(x)= \left( \psi \left( x\right) -\psi \left( a\right)
\right) ^{\delta -1}$, then 
\begin{equation*}
I_{a+}^{\alpha ;\psi }f(x)=\frac{\Gamma \left( \delta \right) }{\Gamma \left(
\alpha +\delta \right) }\left( \psi \left( x\right) -\psi \left( a\right)
\right) ^{\alpha +\delta -1}
\end{equation*}

\item If $g(x)=\left( \psi \left( b\right) -\psi \left( x\right)
\right) ^{\delta -1}$, then
\begin{equation*}
I_{b-}^{\alpha ;\psi }g(x)=\frac{\Gamma \left( \delta \right) }{\Gamma\left(
\alpha +\delta \right) }\left( \psi \left( b\right) -\psi \left( x\right)
\right) ^{\alpha +\delta -1}
\end{equation*}

\end{enumerate}
\end{lemma}

\begin{proof}
See {\rm \cite{AHMJ}}.
\end{proof}

%%%%%%%%%%%%%%%%%%%
Here we evoke two definitions of fractional derivatives with respect to another function, both definitions being motivated by the fractional derivative of Riemann-Liouville and Caputo, in that order, choosing a specific function $\psi$.

\begin{definition}{\rm \cite{AHMJ}} Let $\psi'(x)\neq 0$ $(-\infty\leq a<x<b\leq \infty)$ and $\alpha> 0$, $n\in\mathbb{N}$. The Riemann-Liouville derivatives of a function $f$ with respect to $\psi$ of order $\alpha$ correspondent to the Riemann-Liouville, are defined by
\begin{eqnarray}\label{D2}
\mathcal{D}_{a+}^{\alpha ;\psi }f\left( x\right) &=&\left( \frac{1}{\psi ^{\prime
}\left( x\right) }\frac{d}{dx}\right) ^{n}I_{a+}^{n-\alpha ;\psi }f\left(
x\right)  \notag \\
&=&\frac{1}{\Gamma \left( n-\alpha \right) }\left( \frac{1}{\psi ^{\prime
}\left( x\right) }\frac{d}{dx}\right) ^{n}\int_{a}^{x}\psi ^{\prime }\left(
t\right) \left( \psi \left( x\right) -\psi \left( t\right) \right)
^{n-\alpha -1}f\left( t\right) dt \notag \\
\end{eqnarray}
and
\begin{eqnarray}\label{D3}
\mathcal{D}_{b-}^{\alpha ;\psi }f\left( x\right) &=&\left( -\frac{1}{\psi ^{\prime
}\left( x\right) }\frac{d}{dx}\right) ^{n}I_{a+}^{n-\alpha ;\psi }f\left(
x\right)  \notag \\
&=&\frac{1}{\Gamma \left( n-\alpha \right) }\left( -\frac{1}{\psi ^{\prime
}\left( x\right) }\frac{d}{dx}\right) ^{n}\int_{x}^{b}\psi ^{\prime }\left(
t\right) \left( \psi \left( t\right) -\psi \left( x\right) \right)
^{n-\alpha -1}f\left( t\right) dt,\notag \\
\end{eqnarray}
where $n=[\alpha]+1$.
\end{definition}

\begin{definition}{\rm \cite{RCA}} Let $\alpha>0$, $n\in\mathbb{N}$, $I=[a,b]$ is the interval $-\infty\leq a<b\leq\infty$, $f,\psi\in C^{n}([a,b],\mathbb{R})$ two functions such that $\psi$ is increasing and $\psi'(x)\neq 0$, for all $x\in I$. The left $\psi$-Caputo fractional derivative of $f$ of order $\alpha$ is given by
\begin{equation}\label{D4}
^{C}D_{a+}^{\alpha ;\psi }f\left( x\right) =I_{a+}^{n-\alpha ;\psi }\left( 
\frac{1}{\psi ^{\prime }\left( x\right) }\frac{d}{dx}\right) ^{n}f\left(
x\right)
\end{equation}
and the right $\psi$-Caputo fractional derivative of $f$ by
\begin{equation}\label{D5}
^{C}D_{b-}^{\alpha ;\psi }f\left( x\right) =I_{b-}^{n-\alpha ;\psi }\left( -%
\frac{1}{\psi ^{\prime }\left( x\right) }\frac{d}{dx}\right) ^{n}f\left(
x\right)
\end{equation}
where $n=[\alpha]+1$ for $\alpha\notin \mathbb{N}$ and $n=\alpha$ for $\alpha\in\mathbb{N}$.
\end{definition}

\begin{lemma}\label{LE2} Let $\alpha>0$ and $\delta>0$.
\begin{enumerate}
\item If $f(x)= \left( \psi \left( x\right) -\psi \left( a\right)
\right) ^{\delta -1}$, then 
\begin{equation*}
\mathcal{D}_{a+}^{\alpha ;\psi }f(x)=\frac{\Gamma \left( \delta \right) }{\Gamma
\left( \delta -\alpha \right) }\left( \psi \left( x\right) -\psi \left(
a\right) \right) ^{\alpha +\delta -1}.
\end{equation*}

\item If $g(x)=\left( \psi \left( b\right) -\psi \left( x\right)
\right) ^{\delta -1}$, then
\begin{equation*}
\mathcal{D}_{b-}^{\alpha ;\psi }g(x)=\frac{\Gamma \left( \delta \right) }{\Gamma
\left( \delta -\alpha \right) }\left( \psi \left( b\right) -\psi \left(
x\right) \right) ^{\alpha +\delta -1}.
\end{equation*}

\end{enumerate}
\end{lemma}

\begin{proof}
See {\rm \cite{AHMJ}}.
\end{proof}

%%%%%%%%%%%%%%%%%%%%%%%%%%%%
\begin{theorem}\label{TEO1} If $f\in C^{n}[a,b]$ and $\alpha>0$, then
\begin{equation*}
^{C}D_{a+}^{\alpha ;\psi }f\left( x\right) =\mathcal{D}_{a+}^{\alpha ;\psi }f\left(
x\right) \left[ f\left( x\right) -\overset{n-1}{\underset{k=0}{\sum }}\frac{1%
}{k!}\left( \psi \left( x\right) -\psi \left( a\right) \right) ^{k}f_{\psi
}^{\left[ k\right] }\left( a\right) \right] 
\end{equation*}
and
\begin{equation*}
^{C}D_{b-}^{\alpha ;\psi }f\left( x\right) =\mathcal{D}_{b-}^{\alpha ;\psi }f\left(
x\right) \left[ f\left( x\right) -\overset{n-1}{\underset{k=0}{\sum }}\frac{1%
}{k!}\left( \psi \left( b\right) -\psi \left( x\right) \right) ^{k}f_{\psi
}^{\left[ k\right] }\left( b\right) \right] .
\end{equation*}
\end{theorem}

\begin{proof}
See {\rm \cite{RCA}}.
\end{proof}

%%%%%%%%%%%%%%%%%%%%%%%%%%%%%%%%%%%%%%%%%%%%%%%%%%%%%%%%%%%%%%%%%%%%%%%%%%%%%%%%%%%%%%%%%%%%%%%%%%%%%%%%%%%%%%%%%%%%%%%%%%%%%%%%%%%%%%%%%%%%%%%%%%%%%%%%
\section{$\psi$-Hilfer fractional derivative}

From the definition of fractional derivative in the Riemann-Liouville sense and the Caputo sense \cite{AHMJ}, was introduce the Hilfer fractional derivative \cite{SAM,SRI}, which unifies both derivatives. Motivated by the definition of Hilfer, in this section we present our main result, the so-called $\psi$-Hilfer fractional derivative of a function $f$ with respect to another function. From the fractional derivative $\psi$-Hilfer, we evoke some relations between the $\psi$-fractional integral and the fractional derivative $\psi$-Hilfer, which is a limited operator. In this sense, we study the law of exponents and other important results of the fractional calculus.

\begin{definition} Let $n-1<\alpha <n$ with $n\in\mathbb{N}$, $I=[a,b]$ is the interval such that $-\infty\leq a<b\leq\infty$ and $f,\psi\in C^{n}([a,b],\mathbb{R})$ two functions such that $\psi$ is increasing and $\psi'(x)\neq 0$, for all $x\in I$. The  $\psi$-Hilfer fractional derivative $(\mbox{left-sided and right-sided})$ $^{H}\mathbb{D}_{a+}^{\alpha ,\beta ;\psi }\left( \cdot\right)$ and $^{H}\mathbb{D}_{b-}^{\alpha ,\beta ;\psi }\left( \cdot\right)$ of function of order $\alpha$ and type $0\leq \beta \leq 1$, are defined by
\begin{equation}\label{HIL}
^{H}\mathbb{D}_{a+}^{\alpha ,\beta ;\psi }f\left( x\right) =I_{a+}^{\beta \left(
n-\alpha \right) ;\psi }\left( \frac{1}{\psi ^{\prime }\left( x\right) }\frac{d}{dx}\right) ^{n}I_{a+}^{\left( 1-\beta \right) \left( n-\alpha
\right) ;\psi }f\left( x\right)
\end{equation}
and
\begin{equation}\label{HIL1}
^{H}\mathbb{D}_{b-}^{\alpha ,\beta ;\psi }f\left( x\right) =I_{b-}^{\beta
\left( n-\alpha \right) ;\psi }\left( -\frac{1}{\psi ^{\prime }\left(
x\right) }\frac{d}{dx}\right) ^{n}I_{b-}^{\left( 1-\beta \right) \left(
n-\alpha \right) ;\psi }f\left( x\right).
\end{equation}

The $\psi$-Hilfer fractional derivative as above defined, can be written in the following form
\begin{equation}\label{HIL2}
^{H}\mathbb{D}_{a+}^{\alpha ,\beta ;\psi }f\left( x\right) =I_{a+}^{\gamma -\alpha ;\psi }\mathcal{D}_{a+}^{\gamma ;\psi }f\left( x\right) 
\end{equation}
and
\begin{equation}\label{HIL3}
^{H}\mathbb{D}_{b-}^{\alpha ,\beta ;\psi }f\left( x\right) =I_{b-}^{\gamma -\alpha ;\psi }\left( -1\right) ^{n}\mathcal{D}_{b-}^{\gamma ;\psi }f\left( x\right),
\end{equation}
with $\gamma =\alpha +\beta \left( n-\alpha \right) $ and $I^{\gamma-\alpha;\psi}_{a+}(\cdot)$, $D^{\gamma;\psi}_{a+}(\cdot)$, $I^{\gamma-\alpha;\psi}_{b-}(\cdot)$, $D^{\gamma;\psi}_{b-}(\cdot)$ as defined in {\rm Eq.(\ref{A})}, {\rm Eq.(\ref{D2})}, {\rm Eq.(\ref{I3})} and {\rm Eq.(\ref{D3})}.
\end{definition}

To simplify the notation and the prove of some results, we will introduce the following notation:
\begin{equation*}
f_{\psi +}^{\left[ n\right] }f(x) :=\left( \frac{1}{\psi ^{\prime
}\left( x\right) }\frac{d}{dx}\right) ^{n}f\left( x\right) \text{ and  }%
f_{\psi -}^{\left[ n\right] }f(x) :=\left( -\frac{1}{\psi
^{\prime }\left( x\right) }\frac{d}{dx}\right) ^{n}f\left( x\right) .
\end{equation*}

On the order hand, with this notation we have
\begin{equation*}
\mathcal{D}_{a+}^{\gamma ;\psi }f(x) :=\left( \frac{1}{\psi ^{\prime
}\left( x\right) }\frac{d}{dx}\right) ^{n}I_{a+}^{\left( 1-\beta \right)
\left( n-\alpha \right) ;\psi }f\left( x\right) 
\end{equation*}%
and 
\begin{equation*}
\mathcal{D}_{b-}^{\gamma ;\psi }f\left( x\right) :=\left( -\frac{1}{\psi ^{\prime
}\left( x\right) }\frac{d}{dx}\right) ^{n}I_{b-}^{\left( 1-\beta \right)
\left( n-\alpha \right) ;\psi }f\left( x\right) .
\end{equation*}

Note that, 
\begin{equation*}
f_{\psi +}^{\left[ n\right] }f\left( x\right) =\underset{\alpha \rightarrow
n^{-}}{\lim }\mathcal{D}_{a+}^{\gamma ;\psi }f\left( x\right) \text{ and }f_{\psi -}^{%
\left[ n\right] }f\left( x\right) =\underset{\alpha \rightarrow n^{-}}{\lim }%
\mathcal{D}_{b-}^{\gamma ;\psi }f\left( x\right) .
\end{equation*}

In particular, when $0<\alpha <1$ and $0\leq \beta \leq 1$, we have 
\begin{equation*}
^{H}\mathbb{D}_{a+}^{\alpha ,\beta ;\psi }f\left( x\right) =\frac{1}{\Gamma \left(
\gamma -\alpha \right) }\int_{a}^{x}\left( \psi \left( x\right) -\psi \left(
t\right) \right) ^{\gamma -\alpha -1}\mathcal{D}_{a+}^{\gamma ;\psi }f\left( t\right)
dt,
\end{equation*}
with $\gamma =\alpha +\beta \left( 1-\alpha \right) $ and $\mathcal{D}_{a+}^{\gamma;\psi }\left( \cdot \right) $ is $\psi$-Riemann-Liouville fractional derivative.

\begin{theorem} Suppose that $f,\psi \in C^{n+1}\left[ a,b\right] .$ Then, for all $n-1<\alpha <n$ and $0\leq \beta \leq 1,$ we have
\begin{eqnarray}
^{H}\mathbb{D}_{a+}^{\alpha ,\beta ;\psi }f\left( x\right)  &=&\frac{\left(
\psi \left( x\right) -\psi \left( a\right) \right) ^{\gamma -\alpha }}{%
\Gamma \left( \gamma -\alpha +1\right) }\mathcal{D}_{a+}^{\gamma ;\psi
}\left( a\right)  \label{BA} \\
&&+\frac{1}{\Gamma \left( \gamma -\alpha +1\right) }\int_{a}^{x}\left( \psi
\left( x\right) -\psi \left( t\right) \right) ^{\gamma -\alpha }\frac{d}{dt}%
\mathcal{D}_{a+}^{\gamma ;\psi }f\left( t\right) dt \notag 
\end{eqnarray}%
and 
\begin{eqnarray*}
^{H}\mathbb{D}_{b-}^{\alpha ,\beta ;\psi }f\left( x\right)  &=&\frac{\left(
-1\right) ^{n}\left( \psi \left( b\right) -\psi \left( x\right) \right)
^{\gamma -\alpha }}{\Gamma \left( \gamma -\alpha +1\right) }\mathcal{D}%
_{b-}^{\gamma ;\psi }\left( b\right)  \\
&&-\frac{1}{\Gamma \left( \gamma -\alpha +1\right) }\int_{b}^{b}\left( \psi
\left( t\right) -\psi \left( x\right) \right) ^{\gamma -\alpha }\left(
-1\right) ^{n}\frac{d}{dt}\mathcal{D}_{b-}^{\gamma ;\psi }f\left( t\right)
dt.
\end{eqnarray*}
\end{theorem}

\begin{proof} In fact, integrating by parts {\rm {Eq.(\ref{HIL2})}}, with $u^{\prime }\left(
t\right) =\psi ^{\prime }\left( t\right) \left( \psi \left( x\right) -\psi
\left( t\right) \right) ^{\gamma -\alpha -1}$ and $v\left( t\right)
=\mathcal{D}_{a+}^{\gamma ;\psi }f\left( t\right) ,$ we get 
\begin{eqnarray*}
^{H}\mathbb{D}_{a+}^{\alpha ,\beta ;\psi }f\left( x\right)  &=&\frac{1}{%
\Gamma \left( \gamma -\alpha \right) }\int_{a}^{x}\psi ^{\prime }\left(
t\right) \left( \psi \left( x\right) -\psi \left( t\right) \right) ^{\gamma
-\alpha -1}\mathcal{D}_{a+}^{\gamma ;\psi }f\left( t\right) dt \\
&=&\frac{\left( \psi \left( x\right) -\psi \left( a\right) \right) ^{\gamma
-\alpha }}{\Gamma \left( \gamma -\alpha +1\right) }\mathcal{D}_{a+}^{\gamma
;\psi }\left( a\right)  \\
&&+\frac{1}{\Gamma \left( \gamma -\alpha +1\right) }\int_{a}^{x}\left( \psi
\left( x\right) -\psi \left( t\right) \right) ^{\gamma -\alpha }\frac{d}{dt}%
\mathcal{D}_{a+}^{\gamma ;\psi }f\left( t\right) dt.
\end{eqnarray*}

So we conclude the prove. Similarly, one obtains the other case.
\end{proof}

If $f,\psi \in C^{n+1}\left[ a,b\right] $, using the $f_{\psi +}^{\left[ n \right] }f\left( x\right) =\underset{\alpha \rightarrow n^{-}}{\lim } \mathcal{D}_{a+}^{\gamma ;\psi }f\left( x\right) $ and applying limit on both sides of {\rm Eq.(\ref{BA})}, we obtain 
\begin{equation*}
\underset{\alpha \rightarrow n^{-}}{\lim }\text{ }^{H}\mathbb{D}_{a+}^{\alpha ,\beta
;\psi }f\left( x\right) =f_{\psi +}^{\left[ n\right] }f\left( x\right) \text{
and }\underset{\alpha \rightarrow n^{-}}{\lim }\text{ }^{H}\mathbb{D}_{b-}^{\alpha
,\beta ;\psi }f\left( x\right) =f_{\psi -}^{\left[ n\right] }f\left( x\right).
\end{equation*}

%%%%%%%%%%%%%%%%%%%%%%%%%%%%%%
\begin{theorem} The $\psi$-Hilfer fractional derivatives are bounded operators for all $n-1<\alpha<n$ and $0 \leq\beta \leq 1$, given by
\begin{equation}
\left\Vert ^{H}\mathbb{D}_{a+}^{\alpha ,\beta ;\psi }f\right\Vert _{C_{\gamma ;\psi
}}\leq K\left\Vert f_{\psi }^{\left[ n\right] }\right\Vert _{C_{\gamma ;\psi
}^{n}}
\end{equation}
and
\begin{equation}
\left\Vert ^{H}\mathbb{D}_{b-}^{\alpha ,\beta ;\psi }f\right\Vert _{C_{\gamma ;\psi
}}\leq K\left\Vert f_{\psi }^{\left[ n\right] }\right\Vert _{C_{\gamma ;\psi
}^{n}},
\end{equation}
$K=\dfrac{\left( \psi \left( b\right) -\psi \left( a\right) \right)
^{n-\alpha }}{\left( n-\gamma \right) \left( \gamma -\alpha \right) \Gamma
\left( n-\gamma \right) \Gamma \left( \gamma -\alpha \right) }$.
\end{theorem}

\begin{proof} First, we note that
\begin{eqnarray}
\left\Vert \mathcal{D}_{a_{+}}^{\gamma ;\psi }f\right\Vert _{C_{\gamma ;\psi
}} &=&\left\Vert \left( \frac{1}{\psi ^{\prime }\left( x\right) }\frac{d}{dx}%
\right) ^{n}I_{a+}^{n-\gamma ;\psi }f\right\Vert _{C_{\gamma ;\psi }}  \notag
\label{H} \\
&=&\underset{x\in \left[ a,b\right] }{\max }\left\vert \left( \psi \left(
x\right) -\psi \left( a\right) \right) ^{\gamma }\left( \frac{1}{\psi
^{\prime }\left( x\right) }\frac{d}{dx}\right) ^{n}I_{a+}^{n-\gamma ;\psi
}f\right\vert   \notag \\
&=&\underset{x\in \left[ a,b\right] }{\max }\left\vert 
\begin{array}{c}
\left( \psi \left( x\right) -\psi \left( a\right) \right) ^{\gamma }\left( 
\frac{1}{\psi ^{\prime }\left( x\right) }\frac{d}{dx}\right) ^{n}\frac{1}{%
\Gamma \left( n-\gamma \right) } \\ 
\times \int_{a}^{x}\psi ^{\prime }\left( t\right) \left( \psi \left(
x\right) -\psi \left( t\right) \right) ^{n-\gamma -1}f\left( t\right) dt%
\end{array}%
\right\vert   \notag \\
&\leq &\frac{\left\Vert f_{\psi }^{n}\right\Vert _{C_{\gamma ;\psi }^{n}}}{%
\Gamma \left( n-\gamma \right) }\underset{x\in \left[ a,b\right] }{\max }%
\left\vert \int_{a}^{x}\psi ^{\prime }\left( t\right) \left( \psi \left(
x\right) -\psi \left( t\right) \right) ^{n-\gamma -1}dt\right\vert   \notag
\\
&\leq &\frac{\left( \psi \left( b\right) -\psi \left( a\right) \right)
^{n-\gamma }\left\Vert f_{\psi }^{n}\right\Vert _{C_{\gamma ;\psi }^{n}}}{%
\left( n-\gamma \right) \Gamma \left( n-\gamma \right) },
\end{eqnarray}
with $\gamma=\alpha+\beta(n-\alpha)$.

Using the definition of $\psi$-Hilfer fractional derivative and {\rm Eq.(\ref{H})}, we obtain
\begin{eqnarray*}
\left\Vert ^{H}\mathbb{D}_{a+}^{\alpha ,\beta ;\psi }f\right\Vert
_{C_{\gamma ;\psi }} &=&\left\Vert I_{a+}^{\gamma -\alpha ;\psi }\mathcal{D}%
_{a_{+}}^{\gamma ;\psi }f\right\Vert _{C_{\gamma ;\psi }}  \notag \\
&=&\underset{x\in \left[ a,b\right] }{\max }\left\vert \left( \psi \left(
x\right) -\psi \left( a\right) \right) ^{\gamma }I_{a+}^{\gamma -\alpha
;\psi }\mathcal{D}_{a_{+}}^{\gamma ;\psi }f\left( x\right) \right\vert  
\notag \\
&\leq &\frac{\left\Vert \mathcal{D}_{a_{+}}^{\gamma ;\psi }f\right\Vert
_{C_{\gamma ;\psi }}}{\Gamma \left( \gamma -\alpha \right) }\underset{x\in %
\left[ a,b\right] }{\max }\left\vert \int_{a}^{x}\psi ^{\prime }\left(
t\right) \left( \psi \left( x\right) -\psi \left( t\right) \right) ^{\gamma
-\alpha -1}dt\right\vert   \notag \\
&\leq &\frac{\left( \psi \left( b\right) -\psi \left( a\right) \right)
^{\gamma -\alpha }}{\left( \gamma -\alpha \right) \Gamma \left( \gamma
-\alpha \right) }\left\Vert D_{a_{+}}^{\gamma ;\psi }f\right\Vert
_{C_{\gamma ;\psi }}  \notag \\
&\leq &\frac{\left( \psi \left( b\right) -\psi \left( a\right) \right)
^{n-\alpha }}{\left( n-\gamma \right) \left( \gamma -\alpha \right) \Gamma
\left( n-\gamma \right) \Gamma \left( \gamma -\alpha \right) }\left\Vert
f_{\psi }^{n}\right\Vert _{C_{\gamma ;\psi }^{n}}  \notag \\
&=&K\left\Vert f_{\psi }^{n}\right\Vert _{C_{\gamma ;\psi }^{n}},
\end{eqnarray*}
with $K=\dfrac{\left( \psi \left( b\right) -\psi \left( a\right) \right) ^{n-\alpha }}{\left( n-\gamma \right) \left( \gamma -\alpha \right) \Gamma \left( n-\gamma \right) \Gamma \left( \gamma -\alpha \right)}$.
\end{proof}
%%%%%%%%%%%%%%%%%%%%%%%%%%%%%%%%%%%%%

\begin{remark}\label{H1} Consider the $\psi$-Hilfer fractional derivative and the following function
$g\left( x\right) =I_{a+}^{\left( 1-\beta \right) \left( n-\alpha \right);\psi }f\left( x\right) $, so we have
\begin{equation*}
^{H}\mathbb{D}_{a+}^{\alpha ,\beta ;\psi }f\left( x\right) =I_{a+}^{n-\mu ;\psi
}\left( \frac{1}{\psi ^{\prime }\left( x\right) }\frac{d}{dx}\right) ^{n}g\left( x\right) ,
\end{equation*}
with $\mu = n(1-\beta)+\beta\alpha$.

Thus, we have the following relationship between $\psi$-Hilfer and $\psi$-Caputo fractional derivative, given by
\begin{eqnarray*}
^{H}\mathbb{D}_{a+}^{\alpha ,\beta ;\psi }f\left( x\right)  &=&^{C}D_{a+}^{\mu ;\psi
}g\left( x\right)   \notag \\
&=&^{C}D_{a+}^{\mu ;\psi }\left[ I_{a+}^{\left( 1-\beta \right) \left(
n-\alpha \right) ;\psi }f\left( x\right) \right] .
\end{eqnarray*}
\end{remark}

%%%%%%%%%%%%%%%%%%%%%%%%%%%%%%%

\begin{theorem} Let $n-1<\alpha<n$, $n\in\mathbb{N}$ and $0\leq \beta \leq 1$. If $f\in C^{n}[a,b] $, then
\begin{equation*}
^{H}\mathbb{D}_{a+}^{\alpha ,\beta ;\psi }f\left( x\right) =\mathcal{D}%
_{a+}^{n-\beta \left( n-\alpha \right) ;\psi }\left[ I_{a+}^{\left( 1-\beta
\right) \left( n-\alpha \right) ;\psi }f\left( x\right) -\underset{k=0}{%
\overset{n-1}{\sum }}\frac{\left( \psi \left( x\right) -\psi \left( a\right)
\right) ^{k}\text{ }\mathcal{D}_{a+,k}^{\gamma ;\psi }f\left( a\right) }{k!}%
\right] 
\end{equation*}%
and 
\begin{equation*}
^{H}\mathbb{D}_{b-}^{\alpha ,\beta ;\psi }f\left( x\right) =\mathcal{D}%
_{b-}^{n-\beta \left( n-\alpha \right) ;\psi }\left[ I_{b-}^{\left( 1-\beta
\right) \left( n-\alpha \right) ;\psi }f\left( x\right) -\underset{k=0}{%
\overset{n-1}{\sum }}\frac{\left( -1\right) ^{k}\left( \psi \left( b\right)
-\psi \left( x\right) \right) ^{k}\text{ }\mathcal{D}_{b-,k}^{\gamma ;\psi
}f\left( b\right) }{k!}\right] ,
\end{equation*}
$\gamma=\alpha+\beta(k-\alpha)$.
\end{theorem}

\begin{proof} In fact, consider the function $g\left( x\right) =I_{a+}^{\left( 1-\beta \right) \left( n-\alpha \right);\psi }f\left( x\right) $ and $\delta =n-\beta \left( n-\alpha \right) $, using {\rm Remark \ref{H1}} and {\rm Theorem \ref{TEO1}}, we have
\begin{eqnarray*}
^{H}\mathbb{D}_{a+}^{\alpha ,\beta ;\psi }f\left( x\right) 
&=&^{C}D_{a+}^{\delta ;\psi }g\left( x\right)  \\
&=&\mathcal{D}_{a+}^{\delta ;\psi }\left[ g\left( x\right) -\underset{k=0}{%
\overset{n-1}{\sum }}\frac{\left( \psi \left( x\right) -\psi \left( a\right)
\right) ^{k}}{k!}g_{\psi }^{k}\left( a\right) \right]  \\
&=&\mathcal{D}_{a+}^{\delta ;\psi }\left[ I_{a+}^{\left( 1-\beta \right)
\left( n-\alpha \right) ;\psi }f\left( x\right) -\right.  \\
&&\left. -\underset{k=0}{\overset{n-1}{\sum }}\frac{\left( \psi \left(
x\right) -\psi \left( a\right) \right) ^{k}}{k!}\left( \frac{1}{\psi
^{\prime }\left( x\right) }\frac{d}{dx}\right) ^{k}I_{a+}^{\left( 1-\beta
\right) \left( k-\alpha \right) ;\psi }f\left( a\right) \right]  \\
&=&\mathcal{D}_{a+}^{\delta ;\psi }\left[ I_{a+}^{\left( 1-\beta \right)
\left( n-\alpha \right) ;\psi }f\left( x\right) -\underset{k=0}{\overset{n-1}%
{\sum }}\frac{\left( \psi \left( x\right) -\psi \left( a\right) \right) ^{k}%
}{k!}\mathcal{D}_{a+,k}^{\gamma ;\psi }f\left( a\right) \right] .
\end{eqnarray*}
\end{proof}

%%%%%%%%%%%%%%%%%%%%%%%%%%
\begin{theorem}\label{teo5} If $f\in C^{n}[a,b]$, $n-1<\alpha<n$ and $0\leq \beta \leq 1$, then
\begin{equation*}
I_{a+}^{\alpha ;\psi }\text{ }^{H}\mathbb{D}_{a+}^{\alpha ,\beta ;\psi }f\left( x\right) =f\left( x\right) -\overset{n}{\underset{k=1}{\sum }}\frac{\left( \psi \left( x\right) -\psi \left( a\right) \right) ^{\gamma -k}}{\Gamma \left( \gamma -k+1\right) }f_{\psi }^{\left[ n-k\right] }I_{a+}^{\left( 1-\beta \right) \left( n-\alpha \right) ;\psi }f\left( a\right) 
\end{equation*}
and 
\begin{equation*}
I_{b-}^{\alpha ;\psi }\text{ }^{H}\mathbb{D}_{b-}^{\alpha ,\beta ;\psi }f\left( x\right) =f\left( x\right) -\overset{n}{\underset{k=1}{\sum }}\frac{\left( -1\right) ^{k}\left( \psi \left( b\right) -\psi \left( x\right) \right) ^{\gamma -k}}{\Gamma \left( \gamma -k+1\right) }f_{\psi }^{\left[ n-k\right] }I_{b-}^{\left( 1-\beta \right) \left( n-\alpha \right) ;\psi }f\left(
b\right) .
\end{equation*}
\end{theorem}
\begin{proof} Using {\rm Lemma \ref{LE}} and definition of $\psi$-Hilfer fractional derivative, we get
\begin{eqnarray}\label{A1}
I_{a+}^{\alpha ;\psi }\text{ }^{H}\mathbb{D}_{a+}^{\alpha ,\beta ;\psi }f\left( x\right) 
&=&I_{a+}^{\alpha ;\psi }\left( I_{a+}^{\gamma -\alpha ;\psi }\mathcal{D}_{a+}^{\gamma
;\psi }f\left( x\right) \right)   \notag \\
&=&I_{a+}^{\gamma ;\psi }\mathcal{D}_{a+}^{\gamma ;\psi }f\left( x\right) .
\end{eqnarray}

Integrating by parts, we have
\begin{eqnarray*}
I_{a+}^{\gamma ;\psi }\mathcal{D}_{a+}^{\gamma ;\psi }f\left( x\right)  &=&%
\frac{1}{\Gamma \left( \gamma \right) }\int_{a}^{x}\psi ^{\prime }\left(
t\right) \left( \psi \left( x\right) -\psi \left( t\right) \right) ^{\gamma
-1}\mathcal{D}_{a+}^{\gamma ;\psi }f\left( t\right) dt  \notag \\
&=&\frac{1}{\Gamma \left( \gamma \right) }\left( \int_{a}^{x}\left( \psi
\left( x\right) -\psi \left( t\right) \right) ^{\gamma -1}\frac{d}{dt}%
f_{\psi }^{\left[ n-1\right] }I_{a+}^{\left( 1-\beta \right) \left( n-\alpha
\right) ;\psi }f\left( t\right) dt\right)   \notag \\
&=&\frac{1}{\Gamma \left( \gamma -1\right) }\int_{a}^{x}\left( \psi \left(
x\right) -\psi \left( t\right) \right) ^{\gamma -2}\frac{d}{dt}f_{\psi }^{%
\left[ n-2\right] }I_{a+}^{\left( 1-\beta \right) \left( n-\alpha \right)
;\psi }f\left( t\right) dt-  \notag \\
&&-\frac{1}{\Gamma \left( \gamma \right) }\left( \psi \left( x\right) -\psi
\left( a\right) \right) ^{\gamma -1}f_{\psi }^{\left[ n-1\right]
}I_{a+}^{\left( 1-\beta \right) \left( n-\alpha \right) ;\psi }f\left(
a\right)   \notag \\
&=&\frac{1}{\Gamma \left( \gamma -2\right) }\int_{a}^{x}\left( \psi \left(
x\right) -\psi \left( t\right) \right) ^{\gamma -3}\frac{d}{dt}f_{\psi }^{%
\left[ n-3\right] }I_{a+}^{\left( 1-\beta \right) \left( n-\alpha \right)
;\psi }f\left( t\right) dt-  \notag \\
&&-\frac{1}{\Gamma \left( \gamma \right) }\left( \psi \left( x\right) -\psi
\left( a\right) \right) ^{\gamma -1}f_{\psi }^{\left[ n-1\right]
}I_{a+}^{\left( 1-\beta \right) \left( n-\alpha \right) ;\psi }f\left(
a\right) -  \notag \\
&&-\frac{1}{\Gamma \left( \gamma -1\right) }\left( \psi \left( x\right)
-\psi \left( a\right) \right) ^{\gamma -2}f_{\psi }^{\left[ n-2\right]
}I_{a+}^{\left( 1-\beta \right) \left( n-\alpha \right) ;\psi }f\left(
a\right)   \notag \\
&&\cdot   \notag \\
&&\cdot   \notag \\
&&\cdot 
\end{eqnarray*}

\begin{eqnarray*}
&=&\frac{1}{\Gamma \left( \gamma -n\right) }\int_{a}^{x}\left( \psi \left(
x\right) -\psi \left( t\right) \right) ^{\gamma -n-1}I_{a+}^{\left( 1-\beta
\right) \left( n-\alpha \right) ;\psi }f\left( t\right) dt  \notag \\
&&-\underset{k=1}{\overset{n}{\sum }}\frac{\left( \psi \left( x\right) -\psi
\left( a\right) \right) ^{\gamma -k}}{\Gamma \left( \gamma -k+1\right) }%
f_{\psi }^{\left[ n-k\right] }I_{a+}^{\left( 1-\beta \right) \left( n-\alpha
\right) ;\psi }f\left( a\right)   \notag \\
&=&I_{a+}^{\gamma -n;\psi }I_{a+}^{\left( 1-\beta \right) \left( n-\alpha
\right) ;\psi }f\left( x\right)   \notag \\
&&-\underset{k=1}{\overset{n}{\sum }}\frac{\left( \psi \left( x\right) -\psi
\left( a\right) \right) ^{\gamma -k}}{\Gamma \left( \gamma -k+1\right) }%
f_{\psi }^{\left[ n-k\right] }I_{a+}^{\left( 1-\beta \right) \left( n-\alpha
\right) ;\psi }f\left( a\right) 
\end{eqnarray*}
Introducing the parameters $A=\gamma -n=\alpha +\beta \left( n-\alpha \right) -n=\alpha +\beta n-\beta \alpha -n$ and $B=\left( 1-\beta \right) \left( n-\alpha \right) =n-\alpha -\beta n+\beta
\alpha $, we have $A+B=0$. So, using {\rm Lemma \ref{LE}} and by {\rm Eq.(\ref{A1})}, we conclude that
\begin{eqnarray*}
I_{a+}^{\alpha ;\psi }\text{ }^{H}\mathbb{D}_{a+}^{\alpha ,\beta ;\psi
}f\left( x\right)  &=&I_{a+}^{\gamma -n;\psi }I_{a+}^{\left( 1-\beta \right)
\left( n-\alpha \right) ;\psi }f\left( x\right)  \\
&&-\overset{n}{\underset{k=1}{\sum }}\frac{\left( \psi \left( x\right) -\psi
\left( a\right) \right) ^{\gamma -k}}{\Gamma \left( \gamma +1-k\right) }%
f_{\psi }^{\left[ n-k\right] }I_{a+}^{\left( 1-\beta \right) \left( n-\alpha
\right) ;\psi }f\left( a\right)  \\
&=&f\left( x\right) -\overset{n}{\underset{k=1}{\sum }}\frac{\left( \psi
\left( x\right) -\psi \left( a\right) \right) ^{\gamma -k}}{\Gamma \left(
\gamma +1-k\right) }f_{\psi }^{\left[ n-k\right] }I_{a+}^{\left( 1-\beta
\right) \left( n-\alpha \right) ;\psi }f\left( a\right) .
\end{eqnarray*}
\end{proof}
%%%%%%%%%%%%%%%%%%%%%%%%%%

\begin{theorem} Let $f,g\in C^{n}[a,b]$, $\alpha>0$ and $0\leq\beta\leq 1$. Then
\begin{equation}
\text{ }^{H}\mathbb{D}_{a+}^{\alpha ,\beta ;\psi }f\left( x\right) =\text{ }^{H}\mathbb{D}_{a+}^{\alpha ,\beta ;\psi
}g\left( x\right) \Leftrightarrow f\left( x\right) =g\left( x\right) +%
\underset{k=1}{\overset{n}{\sum }}c_{k}\left( \psi \left( x\right) -\psi
\left( a\right) \right) ^{\gamma -k}
\end{equation}
and
\begin{equation}
\text{ }^{H}\mathbb{D}_{b-}^{\alpha ,\beta ;\psi }f\left( x\right) =\text{ }^{H}\mathbb{D}_{b-}^{\alpha ,\beta ;\psi
}g\left( x\right) \Leftrightarrow f\left( x\right) =g\left( x\right) +%
\underset{k=1}{\overset{n}{\sum }}d_{k}\left( \psi \left( b\right) -\psi
\left( x\right) \right) ^{\gamma -k}.
\end{equation}
\end{theorem}

\begin{proof}
Suppose that $\text{ }^{H}\mathbb{D}_{a+}^{\alpha ,\beta ;\psi }f\left( x\right) =\text{ }^{H}\mathbb{D}_{a+}^{\alpha ,\beta ;\psi }g\left( x\right) $, that is $ \text{ }^{H}\mathbb{D}_{a+}^{\alpha ,\beta ;\psi }\left( f\left( x\right)-g\left( x\right) \right) =0 $.

Applying the left integral operator on both sides of this equality and using {\rm Theorem \ref{teo5}}, we get
\begin{equation*}
I_{a+}^{\alpha ;\psi }\text{ }^{H}\mathbb{D}_{a+}^{\alpha ,\beta ;\psi }\left( f\left( x\right)-g\left( x\right) \right) =0
\end{equation*}
that imply
\begin{equation*}
f\left( x\right) -g\left( x\right) -\overset{n}{\underset{k=1}{\sum }}\frac{ \left( \psi \left( x\right) -\psi \left( a\right) \right) ^{\gamma -k}}{ \Gamma \left( \gamma +1-k\right) }\left( f-g\right) _{\psi }^{\left[ n-k \right] }I_{a+}^{\left( 1-\beta \right) \left( n-\alpha \right) ;\psi }\left( f-g\right) \left( a\right) =0.
\end{equation*}

Then, we conclude that,
\begin{equation*}
f\left( x\right) =g\left( x\right) +\overset{n}{\underset{k=1}{\sum }}%
c_{k}\left( \psi \left( x\right) -\psi \left( a\right) \right) ^{\gamma -k},
\end{equation*}
where $c_{k}=\dfrac{\left( f-g\right) _{\psi }^{\left[ n-k\right] }I_{a+}^{\left( 1-\beta \right) \left( n-\alpha \right) ;\psi }\left( f-g\right) \left( a\right) }{\Gamma \left( \gamma +1-k\right) }$.

To prove the reverse, we assume that
\begin{equation}\label{Z}
f\left( x\right) =g\left( x\right) +\overset{n}{\underset{k=1}{\sum }}
c_{k}\left( \psi \left( x\right) -\psi \left( a\right) \right) ^{\gamma -k}.
\end{equation}

Applying the derivative operator $\text{ }^{H}\mathbb{D}^{\alpha,\beta;\beta}_{a+}(\cdot)$ on both sides of the {\rm Eq.(\ref{Z})}, we get
\begin{equation*}
\text{ }^{H}\mathbb{D}_{a+}^{\alpha ,\beta ;\psi }f\left( x\right) =\text{ }^{H}\mathbb{D}_{a+}^{\alpha ,\beta ;\psi }g\left( x\right) +\overset{n}{\underset{k=1}{\sum }}c_{k}\text{ }^{H}\mathbb{D}_{a+}^{\alpha ,\beta ;\psi }\left( \psi \left( x\right) -\psi \left( a\right) \right) ^{\gamma -k}.
\end{equation*}

Using {\rm Eq.(\ref{RE1})},
$\text{ }^{H}\mathbb{D}_{a+}^{\alpha ,\beta ;\psi }\left( \psi \left( x\right) -\psi
\left( a\right) \right) ^{k}=0$, $k=0,1,2,...,n-1$, we conclude
\begin{equation*}
\text{ }^{H}\mathbb{D}_{a+}^{\alpha ,\beta ;\psi }f\left( x\right) =\text{ }^{H}\mathbb{D}_{a+}^{\alpha ,\beta ;\psi}g\left( x\right) .
\end{equation*}

\end{proof}
%%%%%%%%%%%%%%%%
\begin{lemma}\label{lem4} Let $n-1\leq \gamma <n$ and $f\in C_{\gamma}[a,b]$. Then
\begin{equation*}
I_{a+}^{\alpha ;\psi }f\left( a\right) =\underset{x\rightarrow a+}{\lim }I_{a+}^{\alpha ;\psi }f\left( x\right) =0,\text{ }n-1\leq \gamma <\alpha .
\end{equation*}
\end{lemma}

\begin{proof}
Note that, $I_{a+}^{\alpha ;\psi }f\left( x\right) \in C_{\gamma }\left[ a,b\right] $ is bounded {\rm \cite{AHMJ}}. Since $f\in C_{\gamma }\left[ a,b\right] $ then $\left( \psi \left( x\right) -\psi \left( a\right) \right) ^{\gamma }f\left(x\right) $ is continuous on $[a,b]$ and thus
\begin{equation}\label{bacurinho}
\left\vert \left( \psi \left( x\right) -\psi \left( a\right) \right)
^{\gamma }f\left( x\right) \right\vert <M\Rightarrow \left\vert f\left(
x\right) \right\vert <\left\vert \left( \psi \left( x\right) -\psi \left(a\right) \right) ^{-\gamma }\right\vert M,
\end{equation}
$x\in[a,b]$ for some positive constant $M$.

Applying operator $I_{a+}^{\alpha ;\psi }(\cdot)$ on both sides of {\rm Eq.(\ref{bacurinho})} we obtain
\begin{eqnarray}
\left\vert I_{a+}^{\alpha ;\psi }f\left( x\right) \right\vert  &<&\left\vert
I_{a+}^{\alpha ;\psi }\left( \psi \left( x\right) -\psi \left( a\right)
\right) ^{-\gamma }\right\vert M  \notag \\
&=&M\frac{\Gamma \left( n-\gamma \right) }{\Gamma \left( \alpha +n-\gamma
\right) }\left( \psi \left( x\right) -\psi \left( a\right) \right) ^{\alpha
-\gamma }.
\end{eqnarray}

Since $\gamma<\alpha$, the right-hand side $\rightarrow 0$ as  $x\rightarrow a+$, then we get
\begin{equation*}
I_{a+}^{\alpha ;\psi }f\left( a\right) =\underset{x\rightarrow a+}{\lim }I_{a+}^{\alpha ;\psi }f\left( x\right) =0.
\end{equation*}
\end{proof}

\begin{theorem} Let $f\in C^{1}[a,b]$, $\alpha>0$ and $0\leq\beta\leq 1$, we have
\begin{equation*}
^{H}\mathbb{D}_{a+}^{\alpha ,\beta ;\psi }I_{a+}^{\alpha ;\psi }f\left( x\right)
=f\left( x\right) \text{ and }^{H}\mathbb{D}_{b-}^{\alpha ,\beta ;\psi
}I_{b-}^{\alpha ;\psi }f\left( x\right) =f\left( x\right) .
\end{equation*}
\end{theorem}

\begin{proof}
In fact, we have
\begin{eqnarray*}
^{H}\mathbb{D}_{a+}^{\alpha ,\beta ;\psi }I_{a+}^{\alpha ;\psi }f\left( x\right) 
&=&I_{a+}^{\gamma -\alpha;\psi }\left( \frac{1}{\psi ^{\prime }\left( x\right) }%
\frac{d}{dx}\right) ^{n}I_{a+}^{\left( 1-\beta \right) \left( n-\alpha
\right) ;\psi }I_{a+}^{\alpha ;\psi }f\left( x\right)   \notag \\
&=&I_{a+}^{\gamma -\alpha;\psi }\left( \frac{1}{\psi ^{\prime }\left( x\right) }%
\frac{d}{dx}\right) ^{n}I_{a+}^{n-\beta n+\beta \alpha ;\psi }f\left(
x\right)   \notag \\
&=&I_{a+}^{\gamma -\alpha;\psi }\mathcal{D}_{a+}^{\gamma -\alpha ;\psi }f\left( x\right) .
\end{eqnarray*}

Using {\rm Theorem \ref{teo5}} and {\rm Lemma {\rm \ref{lem4}}}, we conclude that
\begin{eqnarray*}
\text{ }^{H}\mathbb{D}_{a+}^{\alpha ,\beta ;\psi }I_{a+}^{\alpha ;\psi }f\left( x\right) 
&=&I_{a+}^{\gamma -\alpha;\psi }\mathcal{D}_{a+}^{\gamma -\alpha ;\psi }f\left( x\right)  
\notag \\
&=&f\left( x\right) -\overset{n}{\underset{k=1}{\sum }}\frac{\left( \psi
\left( x\right) -\psi \left( a\right) \right) ^{\gamma -k}}{\Gamma \left(
\gamma +1-k\right) }f_{\psi }^{\left[ n-k\right] }I_{a+}^{\left( 1-\beta
\right) \left( n-\alpha \right) ;\psi }f\left( a\right)   \notag \\
&=&f\left( x\right) .
\end{eqnarray*}
\end{proof}

The next result concerns the law of the semigroup between operators $I^{\alpha;\psi}_{a+}(\cdot)$ and $^{H}\mathbb{D}^{\alpha,\beta;\psi}_{a+}(\cdot)$.
%%%%%%%%%%%%%%%%%%%%%%%%%%%%%%%%%
\begin{theorem} Let $n-1<\alpha<n$, $n\in\mathbb{N}$ and $0\leq \beta \leq 1$. If $f\in C^{m+n}[a,b] $, $m,n\in\mathbb{N}$, then for all $k\in\mathbb{N}$ we have
\begin{equation*}
\left( I_{a+}^{\alpha ;\psi }\right) ^{k}\left( ^{H}\mathbb{D}_{a+}^{\alpha ,\beta ;\psi }\right) ^{m}f\left( x\right) =\frac{\left( ^{H}\mathbb{D}_{a+}^{\alpha ,\beta ;\psi }\right) ^{m}f\left( c\right) \left( \psi \left( x\right) -\psi \left( a\right) \right) ^{k\alpha }}{\Gamma \left( k\alpha +1\right) }
\end{equation*}
and
\begin{equation*}
\left( I_{b-}^{\alpha ;\psi }\right) ^{k}\left( ^{H}\mathbb{D}_{b-}^{\alpha ,\beta ;\psi }\right) ^{m}f\left( x\right) =\frac{\left( ^{H}\mathbb{D}_{b-}^{\alpha ,\beta ;\psi }\right) ^{m}f\left( c\right) \left( \psi \left( b\right) -\psi \left( x\right) \right) ^{k\alpha }}{\Gamma \left( k\alpha +1\right) },
\end{equation*}
for some $c\in (a,x)$ and $d \in (x,b)$.
\end{theorem}

\begin{proof} Using {\rm Lemma \ref{LE}}, we get
\begin{equation*}
\left( I_{a+}^{\alpha ;\psi }\right) ^{k}=I_{a+}^{\alpha ;\psi }\cdot \cdot \cdot I_{a+}^{\alpha ;\psi }=I_{a+}^{k\alpha ;\psi }.
\end{equation*}

So,
\begin{eqnarray*}
\left( I_{a+}^{\alpha ;\psi }\right) ^{k}\left( ^{H}\mathbb{D}_{a+}^{\alpha ,\beta ;\psi
}\right) ^{m}f\left( x\right)  &=&I_{a+}^{k\alpha ;\psi }\left(
^{H}\mathbb{D}_{a+}^{\alpha ,\beta ;\psi }\right) ^{m}f\left( x\right)   \notag \\
&=&\frac{1}{\Gamma \left( k\alpha \right) }\int_{a}^{x}\psi ^{\prime }\left(
t\right) \left( \psi \left( x\right) -\psi \left( t\right) \right) ^{k\alpha
-1}\left( ^{H}\mathbb{D}_{a+}^{\alpha ,\beta ;\psi }\right) ^{m}f\left( t\right) dt 
\notag \\
&=&\frac{\left( ^{H}\mathbb{D}_{a+}^{\alpha ,\beta ;\psi }\right) ^{m}f\left( c\right) }{%
\Gamma \left( k\alpha \right) }\int_{a}^{x}\psi ^{\prime }\left( t\right)
\left( \psi \left( x\right) -\psi \left( t\right) \right) ^{k\alpha -1}dt 
\notag \\
&=&\frac{\left( ^{H}\mathbb{D}_{a+}^{\alpha ,\beta ;\psi }\right) ^{m}f\left( c\right)
\left( \psi \left( x\right) -\psi \left( a\right) \right) ^{k\alpha }}{%
k\alpha \Gamma \left( k\alpha \right) }  \notag \\
&=&\frac{\left( ^{H}\mathbb{D}_{a+}^{\alpha ,\beta ;\psi }\right) ^{m}f\left( c\right)
\left( \psi \left( x\right) -\psi \left( a\right) \right) ^{k\alpha }}{%
\Gamma \left( k\alpha +1\right) },
\end{eqnarray*}
with $c\in(a,x)$, guaranteed by the mean value theorem for integrals {\rm \cite{COU}}.
\end{proof}

%%%%%%%%%%%%%%%%%%%%%%%%%%%%%%%%%%%%%%%%%%%%%%%%%%%%%%%%%%%%%%%%%%%%%%%%%%%%%%%%%%%%%%%%%%%%%%%%%%%%%%%%%%%%%%%%%%%%%%%%%%%%%%%%%%%%%%%%%%%%%%%%%%%%%%%%%%%%%%%%%%%%%%%%%%%%%%%%%
\section{Miscellaneous results and examples}

The convergence of functions has great importance for mathematics, specially in analysis, functional analysis, distributions theory and others. In this section, using the fractional $\psi$-Hilfer operator and the fractional integral operator, we present some results of uniformly convergent sequence. In addition, we discuss two examples involving the classical Mittag-Leffler function and the function $(\psi(x)-\psi(a))^{\alpha}$.

\begin{theorem}\label{Teo9} Let $\,n-1<\alpha <n$, $I=\left[ a,b\right] $ be a finite or infinite interval and $\psi \in C\left[ a,b\right] $ a increasing function such that $\psi ^{\prime }\left( x\right) \neq 0,$ for all $x\in I$. Assume that $\left( f_{n}\right) _{n=1}^{\infty }$ is a uniformly convergent sequence of continuous functions on $\left[ a,b\right] $. Then we may interchange the fractional integral operator and the limit process, i.e. 
\begin{equation*}
I_{a+}^{\alpha ;\psi }\underset{n\rightarrow \infty }{\lim }f_{n}\left( x\right) =\underset{n\rightarrow \infty }{\lim }I_{a+}^{\alpha ;\psi }f_{n}\left( x\right) .
\end{equation*}

In particular, the sequence of function $\left( I_{a+}^{\alpha ;\psi }f_{n}\right) _{n=1}^{\infty }$ is uniformly convergent.
\end{theorem}

\begin{proof}
We denote the limit of the sequence $\left( f_{n}\right) $ by $f$. It is well known that $f$ is continuous we then find
\begin{eqnarray*}
\left\vert I_{a+}^{\alpha ;\psi }f_{n}\left( x\right) -I_{a+}^{\alpha ;\psi
}f\left( x\right) \right\vert &\leq &\frac{1}{\Gamma \left( \alpha \right) }%
\int_{a}^{x}\psi ^{\prime }\left( t\right) \left( \psi \left( x\right) -\psi
\left( t\right) \right) ^{\alpha -1}\left\vert f_{n}\left( t\right) -f\left(
t\right) \right\vert dt  \notag \\
&\leq &\frac{\left\Vert f_{n}\left( t\right) -f\left( t\right) \right\Vert
_{\infty }}{\Gamma \left( \alpha \right) }\frac{\left( \psi \left( x\right)
-\psi \left( a\right) \right) ^{\alpha }}{\alpha }  \notag \\
&\leq &\frac{\left( \psi \left( b\right) -\psi \left( a\right) \right)
^{\alpha }}{\Gamma \left( \alpha +1\right) }\left\Vert f_{n}\left( t\right)
-f\left( t\right) \right\Vert _{\infty }.
\end{eqnarray*}

As $f_{n}$ is a uniformly convergent sequence, we conclude the proof.
\end{proof}
%%%%%%%%%%%%%%%%%%%%%%%

\begin{theorem}\label{Teo10} Let $n-1<\alpha <n,$ $n\in\mathbb{N}$ and $I=\left[ a,b\right] $ be a finite or infinite interval and $\psi \in C\left[ a,b\right] $ a increasing function such that $\psi ^{\prime }\left(
x\right) \neq 0$ for all $x\in I$. Assume that $\left( f_{k}\right)_{k=1}^{\infty }$ is a uniformly convergent sequence of continuous functions on $\left[ a,b\right] $ and $\mathcal{D}_{a+}^{\alpha ;\psi }f_{k}$ exist for every $k.$ Moreover assume that $\left( \mathcal{D}_{a+}^{\alpha ;\psi }f_{k}\right)
_{k=1}^{\infty }$ converge uniformly on $\left[ a+\varepsilon ,b\right)$ for every $\varepsilon >0$. Then, for every $x\in \left( a,b\right)\ $ we have
\begin{equation*}
\underset{k\rightarrow \infty }{\lim }\mathcal{D}_{a+}^{\alpha ;\psi }f_{k}\left(
x\right) =\mathcal{D}_{a+}^{\alpha ;\psi }\underset{k\rightarrow \infty }{\lim }%
f_{k}\left( x\right).
\end{equation*}
\end{theorem} 

\begin{proof} Using the definition of fractional operator $\mathcal{D}_{a+}^{\alpha ;\psi
}\left( \cdot \right) $ 
\begin{equation*}
\mathcal{D}_{a+}^{\alpha ;\psi }f\left( x\right) =\left( \frac{1}{\psi ^{\prime }\left( x\right) }\frac{d}{dx}\right) ^{n}I_{a+}^{n-\alpha ;\psi }f\left( x\right) 
\end{equation*}
and by {\rm Theorem \ref{Teo9}}, the sequence $\left( I_{a+}^{n-\alpha ;\psi }f_{k}\right) _{k=1}^{\infty }$ converge uniformly, and 
\begin{equation*}
I_{a+}^{n-\alpha ;\psi }\underset{k\rightarrow \infty }{\lim }f_{k}\left(
x\right) =\underset{k\rightarrow \infty }{\lim }I_{a+}^{n-\alpha ;\psi
}f_{k}\left( x\right) .
\end{equation*}

On the other hand, by hypotheses $\mathcal{D}_{a+}^{\alpha ;\psi }f\left(x\right) =\left( \frac{1}{\psi ^{\prime }\left( x\right) }\dfrac{d}{dx} \right) ^{n}I_{a+}^{n-\alpha ;\psi }f\left( x\right) $ is converge uniformly on $\left[ a+\varepsilon ,b\right) $ for every $\varepsilon >0,$ then we obtain 
\begin{eqnarray*}
\underset{k\rightarrow \infty }{\lim }\mathcal{D}_{a+}^{\alpha ;\psi }f_{k}\left( x\right)  &=&\underset{k\rightarrow \infty }{\lim }\left( \frac{1}{\psi ^{\prime }\left( x\right) }\frac{d}{dx}\right) ^{n}I_{a+}^{n-\alpha ;\psi }f_{k}\left( x\right)   \notag \\
&=&\left( \frac{1}{\psi ^{\prime }\left( x\right) }\frac{d}{dx}\right) ^{n}I_{a+}^{n-\alpha ;\psi }\underset{k\rightarrow \infty }{\lim }f_{k}\left( x\right)   \notag \\
&=&\mathcal{D}_{a+}^{\alpha ;\psi }\underset{k\rightarrow \infty }{\lim }f_{k}\left( x\right).
\end{eqnarray*}
\end{proof}
%%%%%%%%%%%%%%%%%%%%%%%

\begin{theorem}\label{Teo11} Let $n-1<\alpha<n$, $n\in \mathbb{N}$, $0\leq \beta \leq 1$ and $I=\left[a,b\right] $ be a finite or infinite interval and $\psi \in C\left[a,b\right]$ a increasing function such that $\psi ^{\prime }\left( x\right) \neq 0$ for all $x\in I$. Assume that $\left( f_{k}\right) _{k=1}^{\infty }$ is a uniformly convergent sequence of continuous functions on $\left[ a,b\right] $ and $^{H}\mathbb{D}_{a+}^{\alpha ,\beta ;\psi }f_{k}$ exist for every $k$. Moreover assume that $\left( ^{H}\mathbb{D}_{a+}^{\alpha ,\beta ;\psi }f_{k}\right)_{k=1}^{\infty }$ converge uniformly on $\left[ a+\varepsilon ,b\right)$ for every $\varepsilon >0$. Then, for every $x\in \left( a,b\right)$ we have
\begin{equation*}
\underset{k\rightarrow \infty }{\lim }\text{ }^{H}\mathbb{D}_{a+}^{\alpha ,\beta
;\psi }f_{k}\left( x\right) =\text{ }^{H}\mathbb{D}_{a+}^{\alpha ,\beta ;\psi }%
\underset{k\rightarrow \infty }{\lim }f_{k}\left( x\right).
\end{equation*}
\end{theorem}

\begin{proof} Using the definition of fractional operator $\mathcal{D}_{a+}^{\alpha ;\psi }\left( \cdot \right)$ 
\begin{equation*}
\mathcal{D}_{a+}^{\alpha ;\psi }f\left( x\right) =\left( \frac{1}{\psi ^{\prime }\left( x\right) }\frac{d}{dx}\right) ^{n}I_{a+}^{n-\alpha ;\psi }f\left( x\right) ,
\end{equation*}
{\rm Theorem \ref{Teo9}} and {\rm Theorem \ref{Teo10}}, we get
\begin{eqnarray*}
\underset{k\rightarrow \infty }{\lim }\text{ }^{H}\mathbb{D}_{a+}^{\alpha ,\beta
;\psi }f_{k}\left( x\right)  &=&\underset{k\rightarrow \infty }{\lim }%
I_{a+}^{\gamma -\alpha ;\psi }\mathcal{D}_{a+}^{\gamma ;\psi }f_{k}\left( x\right)  
\notag \\
&=&I_{a+}^{\gamma -\alpha ;\psi }\underset{k\rightarrow \infty }{\lim }%
\mathcal{D}_{a+}^{\gamma ;\psi }f_{k}\left( x\right)   \notag \\
&=&I_{a+}^{\gamma -\alpha ;\psi }\mathcal{D}_{a+}^{\gamma ;\psi }\underset{%
k\rightarrow \infty }{\lim }f_{k}\left( x\right)   \notag \\
&=&\text{ }^{H}\mathbb{D}_{a+}^{\alpha ,\beta ;\psi }\underset{k\rightarrow \infty }{%
\lim }f_{k}\left( x\right).
\end{eqnarray*}		

\end{proof}
%%%%%%%%%%%%%%%%%%%

\begin{lemma}\label{LE5} Given $\delta\in\mathbb{R}$, consider the functions
$f\left( x\right) =\left( \psi \left( x\right) -\psi \left( a\right) \right) ^{\delta -1}$
and
$g\left( x\right) =\left( \psi \left( b\right) -\psi \left( x\right) \right) ^{\delta -1}$, where $\delta > n$. Then, for $\delta>0$, $n-1<\alpha<n$, $0 \leq \beta\leq 1$,
\begin{equation*}
\text{ }^{H}\mathbb{D}_{a+}^{\alpha ,\beta ;\psi }f\left( x\right) =\frac{\Gamma \left(
\delta \right) }{\Gamma \left( \delta -\alpha \right) }\left( \psi \left(
x\right) -\psi \left( a\right) \right) ^{\delta -\alpha -1}
\end{equation*}
and
\begin{equation*}
\text{ }^{H}\mathbb{D}_{a+}^{\alpha ,\beta ;\psi }g\left( x\right) =\frac{\Gamma \left(
\delta \right) }{\Gamma \left( \delta -\alpha \right) }\left( \psi \left(
b\right) -\psi \left( x\right) \right) ^{\delta -\alpha -1}.
\end{equation*}
\end{lemma}

\begin{proof} Using the {\rm Lemma \ref{LE1}} and {\rm Lemma \ref{LE2}}, we obtain
\begin{eqnarray*}
\text{ }^{H}\mathbb{D}_{a+}^{\alpha ,\beta ;\psi }f\left( x\right)  &=&I_{a+}^{\gamma -\alpha
;\psi }\mathcal{D}_{a+}^{\gamma ;\psi }f\left( x\right)   \notag \\
&=&I_{a+}^{\gamma -\alpha ;\psi }\mathcal{D}_{a+}^{\gamma ;\psi }\left( \psi \left(
x\right) -\psi \left( a\right) \right) ^{\delta -1}  \notag \\
&=&I_{a+}^{\gamma -\alpha ;\psi }\left( \frac{\Gamma \left( \delta \right) }{%
\Gamma \left( \delta -\gamma \right) }\left( \psi \left( x\right) -\psi
\left( a\right) \right) ^{\delta -\gamma -1}\right)   \notag \\
&=&\frac{\Gamma \left( \delta \right) }{\Gamma \left( \delta -\gamma \right) 
}I_{a+}^{\gamma -\alpha ;\psi }\left( \left( \psi \left( x\right) -\psi
\left( a\right) \right) ^{\delta -\gamma -1}\right)   \notag \\
&=&\frac{\Gamma \left( \delta \right) }{\Gamma \left( \delta -\gamma \right) 
}\frac{\Gamma \left( \delta -\gamma \right) }{\Gamma \left( \delta -\gamma
+\gamma -\alpha \right) }\left( \psi \left( x\right) -\psi \left( a\right)
\right) ^{\delta -\gamma +\gamma -\alpha -1}  \notag \\
&=&\frac{\Gamma \left( \delta \right) }{\Gamma \left( \delta -\alpha \right) 
}\left( \psi \left( x\right) -\psi \left( a\right) \right) ^{\delta -\alpha
-1}.
\end{eqnarray*}

Similarly, the same procedure is performed and we get
\begin{equation*}
\text{ }^{H}\mathbb{D}_{a+}^{\alpha ,\beta ;\psi }g\left( x\right) =\frac{\Gamma \left(
\delta \right) }{\Gamma \left( \delta -\alpha \right) }\left( \psi \left(
b\right) -\psi \left( x\right) \right) ^{\delta -\alpha -1}.
\end{equation*}
\end{proof}

%%%%%%%%%%%%%%%%%%%%%

\begin{remark} In particular, given $n\leq k \in \mathbb{N}$ and as $\delta>n$, we have
\begin{equation*}
\text{ }^{H}\mathbb{D}_{a+}^{\alpha ,\beta ;\psi }\left( \psi \left( x\right) -\psi \left( a\right) \right) ^{k}=\frac{k!}{\Gamma \left( k+1-\alpha \right) }\left( \psi \left( x\right) -\psi \left( a\right) \right) ^{k-\alpha }
\end{equation*}
and
\begin{equation*}
\text{ }^{H}\mathbb{D}_{a+}^{\alpha ,\beta ;\psi }\left( \psi \left( b\right) -\psi \left( x\right) \right) ^{k}=\frac{k!}{\Gamma \left( k+1-\alpha \right) }\left( \psi \left( b\right) -\psi \left( x\right) \right) ^{k-\alpha }.
\end{equation*}

On the other hand, for $n>k\in\mathbb{N}_{0}$, we have
\begin{equation}\label{RE1}
\text{ }^{H}\mathbb{D}_{a+}^{\alpha ,\beta ;\psi }\left( \psi \left( x\right) -\psi \left( a\right) \right) ^{k}=\text{ } ^{H}\mathbb{D}_{a+}^{\alpha ,\beta ;\psi }\left( \psi \left(
b\right) -\psi \left( x\right) \right) ^{k}=0,
\end{equation}
since
\begin{equation*}
\text{ }^{H}\mathbb{D}_{a+}^{\alpha ,\beta ;\psi }\left( \psi \left( x\right) -\psi \left( a\right) \right) ^{k}=\frac{k!}{\Gamma \left( k+1-\alpha \right) }\left( \psi \left( x\right) -\psi \left( a\right) \right) ^{k-\alpha }=0.
\end{equation*}
\end{remark}

%%%%%%%%%%%%%%%%%%%%%%%%%%%%%%
\begin{lemma} Given $\lambda>0$, $n-1<\alpha<n$ and $\leq \beta\leq 0$. Consider the functions
$f\left( x\right) =\mathbb{E}_{\alpha }\left( \lambda \left( \psi \left( x\right) -\psi \left( a\right) \right) ^{\alpha }\right) $ and $g\left( x\right) =\mathbb{E}_{\alpha }\left( \lambda \left( \psi \left( b\right) -\psi \left( x\right) \right) ^{\alpha }\right) $, where $\mathbb{E}_{\alpha }\left( \cdot \right) $ is the Mittag-Leffler function with one parameter. Then,
\begin{equation*}
\text{ }^{H}\mathbb{D}_{a+}^{\alpha ,\beta ;\psi }f\left( x\right) =\lambda
f\left( x\right) \text{ and  }^{H}\mathbb{D}_{b-}^{\alpha ,\beta ;\psi
}f\left( x\right) =\lambda f\left( x\right) .
\end{equation*}
\end{lemma}

\begin{proof} Using the definition of Mittag-Leffler function with one parameter and {\rm Lemma \ref{LE5}}, we have
\begin{eqnarray*}
\text{ }^{H}\mathbb{D}_{a+}^{\alpha ,\beta ;\psi }f\left( x\right)  &=&^{H}%
\mathbb{D}_{a+}^{\alpha ,\beta ;\psi }\left[ \mathbb{E}_{\alpha }\left(
\lambda \left( \psi \left( x\right) -\psi \left( a\right) \right) ^{\alpha
}\right) \right]   \notag \\
&=&\overset{\infty }{\underset{k=0}{\sum }}\frac{\lambda ^{k}}{\Gamma \left(
\alpha k+1\right) }^{H}\mathbb{D}_{a+}^{\alpha ,\beta ;\psi }\left( \psi
\left( x\right) -\psi \left( a\right) \right) ^{k\alpha }  \notag \\
&=&\lambda \underset{k=1}{\overset{\infty }{\sum }}\frac{\lambda ^{k-1}}{%
\Gamma \left( \alpha k+1\right) }\frac{\Gamma \left( \alpha k+1\right) }{%
\Gamma \left( \alpha k-\alpha +1\right) }\left( \psi \left( x\right) -\psi
\left( a\right) \right) ^{k\alpha -\alpha }  \notag \\
&=&\lambda \underset{k=1}{\overset{\infty }{\sum }}\frac{\lambda
^{k-1}\left( \psi \left( x\right) -\psi \left( a\right) \right) ^{\left(
k-1\right) \alpha }}{\Gamma \left( \left( k-1\right) \alpha +1\right) } 
\notag \\
&=&\lambda f\left( x\right) .
\end{eqnarray*}
\end{proof}

%%%%%%%%%%%%%%%%%%%%%%%%%%%%%%%%%%%%%%%%%%%%%%%%%%%%%%%%%%%%%%%%%%%%%%%%%%%%%%%%%%%%%%%%%%%%%%%%%%%%%%%%%%%%%%%%%%%%%%%%%%%%%%%%%%%%%%%%%%%%%%%%%%%%%%%%%%%%%%%%%%%%%%%%%%%%%%%%%%
%%%%%%%%%%%%%%%%%%%%%%%%%%%%%%%%%%%%%%%%%%%%%%%%%%%%%%%%%%%%%%%%%%%%%%%%%%%%%%%%%%%%%%%%%%%%%%%%%%%%%%%%%%%%%%%%%%%%%%%%%%%%%%%%%%%%%%%%%%%%%%%%%%%%%%%%%%%%%%%%%%%%%%%%%%%%%%%%%
\section{A wide class of fractional derivatives and integrals}

The fractional calculus over time has become an important tool for the development of new mathematical concepts in the theoretical sense and practical sense. So far, there are a variety of fractional operators, in the integral sense or in the differential sense. However, natural problems become increasingly complex and certain fractional operators presented with the specific kernel are restricted to certain problems. So \cite{AHMJ,SAM} it was proposed a fractional integral operator with respect to another function, that is, to a function $\psi$, making such a general operator, in the sense that it is enough to choose a function $\psi$ and obtain an existing fractional integral operator. In the same way, we have presented, in section 2 the fractional $\psi$-Hilfer derivative with respect to another function. In this section, we present a class of fractional integrals and fractional derivatives, based on the choice of the arbitrary $\psi$ function.

For the class of integrals that will be presented next, we suggest \cite{AHMJ,SAM,RHM,UNT,ECJT,RCA,SRI,UNT2}.

\begin{enumerate}
\item If we consider $\psi \left( x\right) =x$ in Eq.(\ref{A}), we have 
\begin{equation*}
I_{a+}^{\alpha ;x}f\left( x\right) =\frac{1}{\Gamma \left( \alpha \right) }\int_{a}^{x}\left( x-t\right) ^{\alpha -1}f\left( t\right) dt=\text{ } ^{RL}I_{a+}^{\alpha }f\left( x\right), 
\end{equation*}
the Riemann-Liouville fractional integral.

\item If we consider $\psi \left( x\right) =x$ and $a=-\infty $ in Eq.(\ref{A}), we have 
\begin{equation*}
I_{a+}^{\alpha ;x}f\left( x\right) =\frac{1}{\Gamma \left( \alpha \right) }%
\int_{-\infty }^{x}\left( x-t\right) ^{\alpha -1}f\left( t\right) dt= \text{ }^{L}I_{+}^{\alpha }f\left( x\right),
\end{equation*}
the Liouville fractional integral.

\item If we consider $\psi \left( x\right) =x$ and $a=0$ in Eq.(\ref{A}), we have 
\begin{equation*}
I_{a+}^{\alpha ;x}f\left( x\right) =\frac{1}{\Gamma \left( \alpha \right) }\int_{0}^{x}\left( x-t\right) ^{\alpha -1}f\left( t\right) dt= \text{ }^{R}I_{+}^{\alpha }f\left( x\right),
\end{equation*}
the Riemann fractional integral.

\item Choosing $\psi \left( x\right) =\ln x$ and substituting in Eq.(\ref{A}), we have 
\begin{eqnarray*}
I_{a+}^{\alpha ;\ln x}f\left( x\right) &=&\frac{1}{\Gamma \left( \alpha \right) }\int_{a}^{x}\frac{1}{t}\left( \ln x-\ln t\right) ^{\alpha-1}f\left( t\right) dt \\
&=&\frac{1}{\Gamma \left( \alpha \right) }\int_{a}^{x}\left( \ln \frac{x}{t}%
\right) ^{\alpha -1}f\left( t\right) \frac{dt}{t} \\
&=& \text{ } ^{H}I_{a+}^{\alpha }f\left( x\right),
\end{eqnarray*}%
the Hadamard fractional integral.

\item If we consider $\psi \left( x\right) =x^{\sigma }$, $g\left( x\right) =x^{\sigma \eta }f\left( x\right) $ and substituting in Eq.(\ref{A}), we get 
\begin{eqnarray*}
x^{-\sigma \left( \alpha +\eta \right) }I_{a+}^{\alpha ;x^{\sigma }}g\left(
x\right) &=&x^{-\sigma \left( \alpha +\eta \right) }I_{a+}^{\alpha
;x^{\sigma }}\left( x^{\sigma \eta }f\left( x\right) \right) \\
&=&\frac{\sigma x^{-\sigma \left( \alpha +\eta \right) }}{\Gamma \left(
\alpha \right) }\int_{a}^{x}t^{\sigma \eta +\sigma -1}\left( x^{\sigma
}-t^{\sigma }\right) ^{\alpha -1}f\left( t\right) dt \\
&=&\text{ } ^{EK}I_{a+,\sigma ,\eta }^{\alpha }f\left( x\right),
\end{eqnarray*}%
the Erdèlyi-Kober fractional integral.

\item If we consider $\psi \left( x\right) =x^{\sigma }$, $g\left( x\right)
=x^{\sigma \eta }f\left( x\right) ,$ $a=0$ and substituting in Eq.(\ref{A}), we obtain
\begin{eqnarray*}
x^{-\sigma \left( \alpha +\eta \right) }I_{0+}^{\alpha ;x^{\sigma }}g\left(
x\right) &=&x^{-\sigma \left( \alpha +\eta \right) }I_{0+}^{\alpha
;x^{\sigma }}\left( x^{\sigma \eta }f\left( x\right) \right) \\
&=&\frac{\sigma x^{-\sigma \left( \alpha +\eta \right) }}{\Gamma \left(
\alpha \right) }\int_{0}^{x}t^{\sigma \eta +\sigma -1}\left( x^{\sigma
}-t^{\sigma }\right) ^{\alpha -1}f\left( t\right) dt \\
&=&\text{ } ^{E}I_{0+,\sigma ,\eta }^{\alpha }f\left( x\right),
\end{eqnarray*}%
the Erdèlyi fractional integral.

\item If we consider $\psi \left( x\right) =x$, $g\left( x\right) =x^{\eta}f\left( x\right) ,$ $a=0$ and substituting in Eq.(\ref{A}), we obtain
\begin{eqnarray*}
x^{-\left( \alpha +\eta \right) }I_{0+}^{\alpha ;x}g\left( x\right)
&=&x^{-\alpha -\eta }I_{0+}^{\alpha ;x}\left( x^{\eta }f\left( x\right)
\right) \\
&=&\frac{x^{-\alpha -\eta }}{\Gamma \left( \alpha \right) }%
\int_{0}^{x}t^{\eta }\left( x-t\right) ^{\alpha -1}f\left( t\right) dt \\
&=&\text{ } ^{K}I_{0+,\eta }^{\alpha }f\left( x\right),
\end{eqnarray*}%
the Kober fractional integral.

\item If we consider $\psi \left( x\right) =x^{\rho }$, $g\left( x\right) =x^{\rho \eta }f\left( x\right) $ and substituting in Eq.(\ref{A}), we obtain
\begin{eqnarray*}
\frac{x^{\kappa }}{\rho ^{\beta }}I_{a+}^{\alpha ;x^{\rho }}g\left( x\right)
&=&\frac{x^{\kappa }}{\rho ^{\beta }}I_{a+}^{\alpha ;x^{\sigma }}\left(
x^{\rho \eta }f\left( x\right) \right) \\
&=&\frac{x^{\kappa }\rho ^{1-\beta }}{\Gamma \left( \alpha \right) }%
\int_{a}^{x}t^{\rho \eta +\rho -1}\left( x^{\rho }-t^{\rho }\right) ^{\alpha
-1}f\left( t\right) dt \\
&=&\text{ } ^{\rho }\mathcal{I}_{a+,\eta ,\kappa }^{\alpha ,\beta }f\left(
x\right),
\end{eqnarray*}%
the generalized fractional integral that unify another six fractional
integral.

\item Choosing $\psi \left( x\right) =x^{\rho }$ and substituting in Eq.(\ref{A}), we have
\begin{eqnarray*}
\frac{1}{\rho ^{\alpha }}I_{a+}^{\alpha ;x^{\rho }}f\left( x\right) &=&\frac{%
\rho ^{1-\alpha }}{\Gamma \left( \alpha \right) }\int_{a}^{x}t^{\rho
-1}\left( x^{\rho }-t^{\rho }\right) ^{\alpha -1}f\left( t\right) dt \\&=&\text{ } ^{\rho }\mathcal{I}_{a+}^{\alpha }f\left( x\right),
\end{eqnarray*}
the Katugampola fractional integral.

\item Choosing $\psi \left( x\right) =x,$ $g\left( x\right) =\mathbb{E}%
_{\alpha ,\beta }^{\gamma }\left( \omega \left( x-t\right) ^{\alpha }\right)f\left( x\right) $ and substituting in Eq.(\ref{A}), we get 
\begin{eqnarray*}
\Gamma \left( \alpha \right) I_{a+}^{\alpha ;x}g\left( x\right) &=&\Gamma
\left( \alpha \right) I_{a+}^{\alpha ;x}\mathbb{E}_{\alpha ,\beta }^{\gamma
}\left( \omega \left( x-t\right) ^{\alpha }\right) f\left( x\right) \\
&=&\int_{a}^{x}\left( x-t\right) ^{\alpha -1}\mathbb{E}_{\alpha ,\beta
}^{\gamma }\left( \omega \left( x-t\right) ^{\alpha }\right) f\left(
t\right) dt \\
&=&\mathcal{E}_{a+,\alpha ,\beta }^{\omega ,\gamma }f\left( x\right),
\end{eqnarray*}
the Prabhakar fractional integral.

\item If we consider $\psi \left( x\right) =x$, $a=c$ and substituting in Eq.(\ref{A}), we obtain
\begin{eqnarray*}
I_{c+}^{\alpha ;x}f\left( x\right) &=&\frac{1}{\Gamma \left( \alpha \right) }\int_{c}^{x}\left( x-t\right) ^{\alpha -1}f\left( t\right) dt \\&=&\mathcal{I}_{c}^{\alpha }f\left( x\right),
\end{eqnarray*}
the Chen fractional integral.

\item For $\psi \left( x\right) =x$, $a=-\infty ,$ $b=\infty $ and substituting in {Eq.(\ref{A})} and {Eq.(\ref{I3})} respectively, we have
\begin{equation*}
\frac{I_{a+}^{\alpha ;x}f\left( x\right) +I_{b-}^{\alpha ;x}f\left( x\right) }{2\cos \left( \frac{\pi \alpha }{2}\right) }=\frac{_{L}I_{+}^{\alpha
}f\left( x\right) +\text{ }_{L}I_{-}^{\alpha }f\left( x\right) }{2\cos \left( \frac{\pi \alpha }{2}\right) }=\text{ }_{RZ}I^{\alpha }f\left( x\right),
\end{equation*}
the Riesz fractional integral.

\item For $\psi \left( x\right) =x$, $a=-\infty ,$ $b=\infty ,$ $0<\theta <1$ and substituting in {Eq.(\ref{A})} and {Eq.(\ref{I3})} respectively, we get 
\begin{eqnarray*}
C_{-}\left( \theta ,\alpha \right) I_{a+}^{\alpha ;x}f\left( x\right)
+C_{+}\left( \theta ,\alpha \right) I_{b-}^{\alpha ;x}f\left( x\right) 
&=&C_{-}\left( \theta ,\alpha \right) I_{+}^{\alpha }f\left( x\right)
+C_{+}\left( \theta ,\alpha \right) I_{-}^{\alpha }f\left( x\right)  \\
&=&\text{ }_{F}I_{\theta }^{\alpha }f\left( x\right) ,
\end{eqnarray*}
the Feller fractional integral, with 
\begin{equation*}
C_{-}\left( \theta ,\alpha \right) =\dfrac{\sin \left( \frac{\left( \alpha
-\theta \right) \pi }{\alpha }\right) }{\sin \left( \pi \theta \right) }%
\text{ and}\ \text{\ }C_{+}\left( \theta ,\alpha \right) =\dfrac{\sin \left( 
\frac{\left( \alpha +\theta \right) \pi }{\alpha }\right) }{\sin \left( \pi
\theta \right) }.
\end{equation*}
\item For $\psi \left( x\right) =x$, $b=\infty $ and substituting in {Eq.(\ref{I3})}, we obtain
\begin{equation*}
I_{b-}^{\alpha ;x}=\frac{1}{\Gamma \left( \alpha \right) }\int_{x}^{\infty
}\left( t-x\right) ^{\alpha -1}f\left( t\right) dt=\text{ }_{L}I_{-}^{\alpha
}=\text{ }_{x}W_{\infty }^{\alpha }f\left( x\right),
\end{equation*}
the Weyl integral fractional.
\end{enumerate}

On the other hand, using the $\psi$-Hilfer fractional derivative operator Eq.(\ref{HIL}) and Eq.(\ref{HIL1}), we present a wide class of fractional derivatives by choosing $\psi$, $a$, $b$ and taking the limit of the parameters $\alpha$ and $\beta$.

For this wide class of fractional derivatives that will be presented next, we suggest \cite{AHMJ,SAM,RHM,UNT,UNT1,ECJT,RCA,SRI}.

\begin{enumerate}
\item Taking the limit $\beta \rightarrow 1$ on both sides of Eq.(\ref{HIL}), we get
\begin{equation*}
^{H}\mathbb{D}_{a+}^{\alpha ,1;\psi }f\left( x\right) =I_{a+}^{n-\alpha
;\psi }\left( \frac{1}{\psi ^{\prime }\left( x\right) }\frac{d}{dx}\right)
^{n}f\left( x\right) =\text{ }^{C}D_{a+}^{\alpha ;\psi },
\end{equation*}
the $\psi -$Caputo fractional derivative with respect to another function.

\item Taking the limit $\beta \rightarrow 0$ on both sides of Eq.(\ref{HIL}), we get
\begin{equation*}
^{H}\mathbb{D}_{a+}^{\alpha ,0;\psi }f\left( x\right) =\left( \frac{1}{\psi
^{\prime }\left( x\right) }\frac{d}{dx}\right) ^{n}I_{a+}^{\left( n-\alpha
\right) ;\psi }f\left( x\right) = \mathcal{D}_{a+}^{\alpha ;\psi },
\end{equation*}
the $\psi -$Riemann-Liouville fractional derivative with respect to another function.

\item Consider the $\psi \left( x\right) =x$ and taking the limit $\beta \rightarrow 1$ on both sides of Eq.(\ref{HIL}), we get
\begin{eqnarray*}
^{H}\mathbb{D}_{a+}^{\alpha ,1;x}f\left( x\right) &=&I_{a+}^{n-\alpha
;x}\left( \frac{d}{dx}\right) ^{n}f\left( x\right)  \nonumber \\
&=&\frac{1}{\Gamma \left( n-\alpha \right) }\int_{a}^{x}\left( x-t\right)
^{n-\alpha -1}\left( \frac{d}{dx}\right) ^{n}f\left( t\right) dt=\text{ }^{C}
D_{a+}^{\alpha }f\left( x\right),
\end{eqnarray*}
the Caputo fractional derivative.

\item Consider the $\psi \left( x\right) =x^{\rho }$ and taking the limit $\beta \rightarrow 0$ on both sides of Eq.\ref{HIL}), we get
\begin{eqnarray*}
\rho ^{\alpha } \text{ } ^{H}\mathbb{D}_{a+}^{\alpha ,0;x^{\rho }}f\left(
x\right) &=&\rho ^{\alpha }\left( \frac{1}{\rho x^{\rho -1}}\frac{d}{dx}%
\right) ^{n}I_{a+}^{n-\alpha ;x^{\rho }}f\left( x\right)  \nonumber \\
&=&\rho ^{\alpha -n+1}\left( \frac{1}{x^{\rho -1}}\frac{d}{dx}\right) ^{n}%
\frac{1}{\Gamma \left( n-\alpha \right) }\int_{a}^{x}t^{\rho -1}\left(
x^{\rho }-t^{\rho }\right) ^{n-\alpha -1}f\left( t\right) dt  \nonumber \\
&=& \text{ }^{\rho }D_{a+}^{\alpha }f\left( x\right),
\end{eqnarray*}
the Katugampola fractional derivative.

\item For $\psi \left( x\right) =x$ and taking the limit $\beta \rightarrow 0$ on both sides of Eq.(\ref{HIL}), we have
\begin{eqnarray*}
^{H}\mathbb{D}_{a+}^{\alpha ,0;x}f\left( x\right) &=&\left( \frac{d}{%
dx}\right) ^{n}I_{a+}^{n-\alpha ;x}f\left( x\right)  \nonumber \\
&=&\left( \frac{d}{dx}\right) ^{n}\frac{1}{\Gamma \left( n-\alpha \right) }%
\int_{a}^{x}\left( x-t\right) ^{n-\alpha -1}f\left( t\right) dt  \nonumber \\
&=&\mathcal{D}_{a+}^{\alpha }f\left( x\right),
\end{eqnarray*}
the Riemann-Liouville fractional derivative.

\item For $\psi \left( x\right) =\ln x$ and taking the limit $\beta \rightarrow 0$ on both sides of Eq.(\ref{HIL}), we have
\begin{eqnarray*}
^{H}\mathbb{D}_{a+}^{\alpha ,0;\ln x}f\left( x\right) &=&\left( x\frac{d}{dx}\right) ^{n}I_{a+}^{n-\alpha ;\ln x}f\left( x\right)  \nonumber
\\
&=&\left( x\frac{d}{dx}\right) ^{n}\frac{1}{\Gamma \left( n-\alpha \right) }\int_{a}^{x}\left( \ln \frac{x}{t}\right) ^{n-\alpha -1}f\left( t\right) \frac{dt}{t}  \nonumber \\
&=&\text{ }^{H}D_{a+}^{\alpha }f\left( x\right),
\end{eqnarray*}
the Hadamard fractional derivative.

\item For $\psi \left( x\right) =\ln x$ and taking the limit $\beta \rightarrow 1$ on both sides of Eq.(\ref{HIL}), we have
\begin{eqnarray*}
^{H}\mathbb{D}_{a+}^{\alpha ,1;\ln x}f\left( x\right)
&=&I_{a+}^{n-\alpha ;\ln x}\left( x\frac{d}{dx}\right) ^{n}f\left( x\right) 
\nonumber \\
&=&\frac{1}{\Gamma \left( n-\alpha \right) }\int_{a}^{x}\left( \ln \frac{x}{t%
}\right) ^{n-\alpha -1}\left( t\frac{d}{dt}\right) ^{n}f\left( t\right) 
\frac{dt}{t}  \nonumber \\
&=&\text{ }^{CH}D_{a+}^{\alpha }f\left( x\right),
\end{eqnarray*}
the Caputo-Hadamard fractional derivative.

\item Consider the $\psi \left( x\right) =x^{\rho }$ and taking the limit $\beta \rightarrow 1$ on both sides of Eq.(\ref{HIL}), we have
\begin{eqnarray*}
\rho ^{\alpha }\text{ }^{H}\mathbb{D}_{a+}^{\alpha ,1;x^{\rho }}f\left(
x\right) &=&\rho ^{\alpha }I_{a+}^{n-\alpha ;x^{\rho }}\left( \frac{1}{\rho
x^{\rho -1}}\frac{d}{dx}\right) ^{n}f\left( x\right)  \nonumber \\
&=&\frac{\rho ^{\alpha -n+1}}{\Gamma \left( n-\alpha \right) }%
\int_{a}^{x}t^{\rho -1}\left( x^{\rho }-t^{\rho }\right) ^{n-\alpha
-1}\left( \frac{1}{t^{\rho -1}}\frac{d}{dt}\right) ^{n}f\left( t\right) dt 
\nonumber \\
&=&\text{ }^{CK}D_{a+}^{\alpha ,\rho }f\left( x\right),
\end{eqnarray*}
the Caputo-Katugampola (Caputo-type) fractional derivative.

\item Consider the $\psi \left( x\right) =\ln x$ and substituting in Eq.(\ref{D2}), we get
\begin{eqnarray}\label{V1}
D_{a+}^{\gamma ;\ln x}f\left( x\right) &=&\left( x\frac{d}{dx}\right)
^{n}I_{a+}^{n-\gamma ;\ln x}f\left( x\right)  \nonumber \\
&=&\left( x\frac{d}{dx}\right) ^{n}\frac{1}{\Gamma \left( n-\gamma \right) }%
\int_{a}^{x}\left( \ln \frac{x}{t}\right) ^{n-\gamma -1}f\left( t\right) 
\frac{dt}{t}  \nonumber \\
&=&\text{ }^{H}D_{a+}^{\gamma }f\left( x\right),
\end{eqnarray}
the Hadamard fractional derivative.

Now, recover that 
\begin{eqnarray}\label{V2}
^{H}\mathbb{D}_{a+}^{\alpha ,\beta ;\psi }f\left( x\right) &=&I_{a+}^{\gamma
-\alpha ;\psi }\left( \frac{1}{\psi ^{\prime }\left( x\right) }\frac{d}{dx}\right)  ^{n}I_{a+}^{\left( 1-\beta \right) \left( n-\alpha \right) ;\psi }f\left( x\right)  \nonumber \\
&=&I_{a+}^{\gamma -\alpha ;\psi }D_{a+}^{\gamma ;\psi }f\left( x\right),
\end{eqnarray}
with $\gamma =\alpha +\beta \left( n-\alpha \right) $ and $D_{a+}^{\gamma ;\psi }\left( \cdot \right) $ is the Riemann-Liouville fractional derivative with respect to another function. Thus, by Eq.(\ref{V1}) and Eq.(\ref{V2}), we have 
\begin{eqnarray*}
^{H}\mathbb{D}_{a+}^{\alpha ,\beta ;\ln x}f\left( x\right) &=&I_{a+}^{\gamma
-\alpha }\text{ }^{H}D_{a+}^{\gamma }f\left( x\right)  \nonumber \\
&=&\text{ }^{H}D_{a+}^{\alpha ,\beta }f\left( x\right),
\end{eqnarray*}
the Hilfer-Hadamard fractional derivative.

\item Consider the $\psi \left( x\right) =x^{\rho }$ and substituting in Eq.(\ref{D2}), we get
\begin{eqnarray}\label{V3}
\mathcal{D}_{a+}^{\gamma ;x^{\rho }}f\left( x\right) &=&\left( \frac{1}{\rho x^{\rho
-1}}\frac{d}{dx}\right) ^{n}I_{a+}^{n-\gamma ;x^{\rho }}f\left( x\right) 
\nonumber \\
&=&\rho ^{1-n}\left( \frac{1}{x^{\rho -1}}\frac{d}{dx}\right) ^{n}\frac{1}{%
\Gamma \left( n-\gamma \right) }\int_{a}^{x}t^{\rho -1}\left( x^{\rho
}-t^{\rho }\right) ^{n-\gamma -1}f\left( t\right) dt  \nonumber \\
&=&\frac{1}{\rho ^{\alpha }}\text{ }^{\rho }D_{a+}^{\alpha }f\left( x\right),
\end{eqnarray}
the Katugampola fractional derivative.

Now, recover that 
\begin{eqnarray}\label{V4}
^{H}\mathbb{D}_{a+}^{\alpha ,\beta ;\psi }f\left( x\right) &=&I_{a+}^{\gamma
-\alpha ;\psi }\left( \frac{1}{\psi ^{\prime }\left( x\right) }\frac{d}{dx}%
\right) ^{n}I_{a+}^{\left( 1-\beta \right) \left( n-\alpha \right) ;\psi
}f\left( x\right)  \nonumber \\
&=&I_{a+}^{\gamma -\alpha ;\psi }\text{ }\mathcal{D}_{a+}^{\gamma ;\psi }f\left( x\right),
\end{eqnarray}
with $\gamma =\alpha +\beta \left( n-\alpha \right) $ and $\mathcal{D}_{a+}^{\gamma ;\psi }\left( \cdot \right) $ is Riemann-Liouville fractional derivative with respect to another function. Thus, by Eq.(\ref{V3} ) and Eq.(\ref{V4}), we have
\begin{eqnarray*}
^{H}\mathbb{D}_{a+}^{\alpha ,\beta ;x^{\rho }}f\left( x\right) &=&^{\rho }I_{a+}^{\gamma -\alpha }\text{ }^{\rho }D_{a+}^{\gamma }f\left( x\right)  \nonumber \\
&=&\text{ }^{\rho }D_{a+}^{\alpha ,\beta }f\left( x\right),
\end{eqnarray*}
the Hilfer-Katugampola fractional derivative.

\item For $\psi \left( x\right) =x,$ $a=0$ and taking the limit $\beta \rightarrow 0$ on both sides of Eq.(\ref{HIL}), we have
\begin{eqnarray*}
^{H}\mathbb{D}_{0+}^{\alpha ,0;\psi }f\left( x\right) &=&\left( \frac{d}{dx}%
\right) ^{n}I_{0+}^{n-\alpha ;\psi }f\left( x\right)  \nonumber \\
&=&^{R}D_{+}^{\alpha }f\left( x\right),
\end{eqnarray*}
the Riemann fractional derivative.

\item For $\psi \left( x\right) =x,$ $a=c$ and taking the limit $\beta \rightarrow 0$ on both sides of Eq.(\ref{HIL}), we have
\begin{eqnarray*}
^{H}\mathbb{D}_{c+}^{\alpha ,0;\psi }f\left( x\right) &=&\left( \frac{d}{dx}%
\right) ^{n}I_{c+}^{n-\alpha ;\psi }f\left( x\right)  \nonumber \\
&=&\left( \frac{d}{dx}\right) ^{n}\frac{1}{\Gamma \left( n-\alpha \right) }%
\int_{c}^{x}\left( x-t\right) ^{n-\alpha -1}f\left( t\right) dt  \nonumber \\
&=&D_{c}^{\alpha }f\left( x\right),
\end{eqnarray*}
the Chen fractional derivative.

\item Consider the $\psi \left( x\right) =x,$ $a=0$, $g\left( x\right)
=f\left( x\right) -f\left( 0\right) $ and taking the limit $\beta
\rightarrow 0$ on both sides of Eq.(\ref{HIL}), we get 
\begin{eqnarray*}
^{H}\mathbb{D}_{0+}^{\alpha ,0;\psi }g\left( x\right) &=&^{H}\mathbb{D}%
_{0+}^{\alpha ,0;\psi }\left( f\left( x\right) -f\left( 0\right) \right) 
\nonumber \\
&=&\left( \frac{d}{dx}\right) ^{n}I_{0}^{n-\alpha ;\psi }\left( f\left(
x\right) -f\left( 0\right) \right)  \nonumber \\
&=&D_{x}^{\alpha }f\left( x\right),
\end{eqnarray*}
the Jumarie fractional derivative.

\item For $\psi \left( x\right) =x,$ $g\left( x\right) =\mathbb{E}_{\rho ,n-\alpha }^{-\gamma }\left[ \omega \left( x-t\right) ^{\rho }\right] f\left( x\right)  $ and taking the limit $\beta \rightarrow 0$ on both sides of Eq.(\ref{HIL}), we get
\begin{eqnarray*}
\Gamma \left( n-\alpha \right) ^{H}\mathbb{D}_{a+}^{\alpha ,0;x}g\left(
x\right) &=&\Gamma \left( n-\alpha \right) ^{H}\mathbb{D}_{a+}^{\alpha
,0;x}\left( \mathbb{E}_{\rho ,n-\alpha }^{-\gamma }\left[ \omega \left( x-t\right)
^{\rho }\right] f\left( x\right) \right)  \nonumber \\
&=&\Gamma \left( n-\alpha \right) \left( \frac{d}{dx}\right)
^{n}I_{a+}^{n-\alpha ;\psi }\left( \mathbb{E}_{\rho ,n-\alpha }^{-\gamma }\left[
\omega \left( x-t\right) ^{\rho }\right] f\left( x\right) \right)  \nonumber
\\
&=&\left( \frac{d}{dx}\right) ^{n}\int_{a}^{x}\left( x-t\right) ^{n-\alpha
-1}\mathbb{E}_{\rho ,n-\alpha }^{-\gamma }\left[ \omega \left( x-t\right) ^{\rho }%
\right] f\left( t\right) dt  \nonumber \\
&=&D_{a+,\gamma ,\alpha }^{\omega ,\rho }f\left( x\right),
\end{eqnarray*}
the Prabhakar fractional derivative.

\item First, taking the limit $\beta \rightarrow 1$ on both sides of Eq.(\ref{HIL}), we have
\begin{equation*}
^{H}\mathbb{D}_{a+}^{\alpha ,1;\psi }f\left( x\right) =I_{a+}^{n-\alpha
;\psi }\left( \frac{1}{\psi ^{\prime }\left( x\right) }\frac{d}{dx}\right)
^{n}f\left( x\right) =\text{ }^{C}D_{a+}^{\alpha ;\psi }f\left( x\right),
\end{equation*}
where $^{C}D_{a+}^{\alpha ;\psi }f\left( x\right) $ is a Caputo fractional derivative with respect to another function.

On the order hand for $\psi \left( x\right) =x^{\sigma }$ and $g\left(
x\right) =x^{\sigma \left( \eta +\alpha \right) }f\left( x\right) $, we get
\begin{equation*}
^{EK}D_{a+,\sigma ,\eta }^{\alpha }f\left( x\right) =x^{-\sigma \eta }\text{ 
}^{C}D_{a+}^{\alpha ;\psi }\left( x^{\sigma \left( \eta +\alpha \right)
}f\left( x\right) \right).
\end{equation*}

Thus, we conclude that
\begin{eqnarray*}
x^{-\sigma \eta }\text{ }^{H}\mathbb{D}_{a+}^{\alpha ,1;\psi }g\left(
x\right) &=&x^{-\sigma \eta }\text{ }^{H}\mathbb{D}_{a+}^{\alpha ,1;\psi
}\left( x^{\sigma \left( \eta +\alpha \right) }f\left( x\right) \right) 
\nonumber \\
&=&^{EK}D_{a+,\sigma ,\eta }^{\alpha }f\left( x\right),
\end{eqnarray*}
the Erdèlyi-Kober fractional derivative.
 
\item If we consider $\psi \left( x\right) =x$, $a=-\infty $ and taking the limit $ \beta \rightarrow 0$ on both sides of Eq.(\ref{HIL1} ), we obtain
\begin{eqnarray*}
^{H}\mathbb{D}_{-\infty }^{\alpha ,0;x}f\left( x\right) &=&\left(- \frac{d}{dx}\right)
^{n}I_{-\infty }^{n-\alpha ;x}f\left( x\right)  \notag \\
&=&\left( -\frac{d}{dx}\right) ^{n}\frac{1}{\Gamma \left( n-\alpha \right) }%
\int_{-\infty }^{x}\left( x-t\right) ^{n-\alpha -1}f\left( t\right) dt 
\notag \\
&=&\text{ }_{L}\mathcal{D}_{+}^{\alpha }f\left( x\right),
\end{eqnarray*}
the Liouville fractional derivative.

\item If we consider $\psi \left( x\right) =x$, $a=-\infty $ and taking the limit $\beta \rightarrow 1$ on both sides of Eq.(\ref{HIL1} ), we obtain,
\begin{eqnarray*}
^{H}\mathbb{D}_{-\infty }^{\alpha ,1;x}f\left( x\right) &=&I_{-\infty }^{n-\alpha
;x}\left(- \frac{d}{dx}\right) ^{n}f\left( x\right)  \notag \\
&=&\frac{1}{\Gamma \left( n-\alpha \right) }\int_{-\infty }^{x}\left(
x-t\right) ^{n-\alpha -1}\left( -\frac{d}{dx}\right) ^{n}f\left( t\right) dt 
\notag \\
&=&\text{ }_{LC}D_{+}^{\alpha }f\left( x\right),
\end{eqnarray*}
the Liouville-Caputo fractional derivative.

\item If we consider $\psi \left( x\right) =x$, $a=-\infty ,$ $b=\infty $ and taking the limit $\beta \rightarrow 0$ on both sides of Eq.(\ref{HIL} ) and Eq.(\ref{HIL1} ) respectively, we obtain
\begin{equation*}
\frac{-\left( \text{ }^{H}\mathbb{D}_{-\infty }^{\alpha ,0;x}f\left( x\right) +\text{ 
}^{H}\mathbb{D}_{\infty }^{\alpha ,0;x}f\left( x\right) \right) }{2\cos \left( \frac{%
\pi \alpha }{2}\right) }=\frac{-\left( \text{ }_{L}\mathcal{D}_{+}^{\alpha }f\left(
x\right) +\text{ }_{L}\mathcal{D}_{-}^{\alpha }f\left( x\right) \right) }{2\cos \left( 
\frac{\pi \alpha }{2}\right) }=\text{ }_{RZ}D^{\alpha }f\left( x\right),
\end{equation*}
the Riesz fractional derivative.

\item If we consider $\psi \left( x\right) =x$, $a=-\infty ,$ $b=\infty ,$ $0<\theta <1$ and taking the limit $\beta \rightarrow 0$ on both sides of Eq.(\ref{HIL}) and Eq.(\ref{HIL1}) respectively, we obtain
\begin{eqnarray*}
&&-\left( C_{+}\left( \theta ,\alpha \right) \text{ }^{H}\mathbb{D}_{-\infty
}^{\alpha ,0;x}f\left( x\right) +C_{-}\left( \theta ,\alpha \right) \text{ }%
^{H}\mathbb{D}_{\infty }^{\alpha ,0;x}f\left( x\right) \right)  \\
&=&-\left( C_{+}\left( \theta ,\alpha \right) \text{ }_{L}\mathcal{D}%
_{+}^{\alpha }f\left( x\right) +C_{-}\left( \theta ,\alpha \right) \text{ }%
_{L}\mathcal{D}_{-}^{\alpha }f\left( x\right) \right)  \\
&=&\text{ }_{F}D_{\theta }^{\alpha }f\left( x\right) ,
\end{eqnarray*}
the Feller fractional derivative with 
\begin{equation*}
C_{-}\left( \theta ,\alpha \right) =\dfrac{\sin \left( \frac{\left( \alpha
-\theta \right) \pi }{\alpha }\right) }{\sin \left( \pi \theta \right) }%
\text{ and$\ $\ }C_{+}\left( \theta ,\alpha \right) =\dfrac{\sin \left( 
\frac{\left( \alpha +\theta \right) \pi }{\alpha }\right) }{\sin \left( \pi
\theta \right) }.
\end{equation*}

\item If we consider $\psi \left( x\right) =x,$ $b=\infty $ and taking limit $\beta \rightarrow 0$ on both sides of the Eq.(\ref{HIL1}), we obtain
\begin{eqnarray*}
\text{ }^{H}\mathbb{D}_{\infty }^{\alpha ,0;x}f\left( x\right) &=&\left( -1\right)
^{n}\left( \frac{d}{dx}\right) ^{n}\frac{1}{\Gamma \left( n-\alpha \right) }%
\int_{x}^{\infty }\left( x-t\right) ^{n-\alpha -1}f\left( t\right) dt  \notag
\\
&=&\text{ }_{L}\mathcal{D}_{-}^{\alpha }f\left( x\right) =\text{ }_{x}D_{\infty
}^{\alpha }f\left( x\right),
\end{eqnarray*}
the Weyl fractional derivative.

\item If we consider $\psi \left( x\right) =x,$ $b=\infty $ and taking the limit $\beta \rightarrow 0$ on both sides of Eq.(\ref{HIL1}), we obtain
\begin{eqnarray*}
^{H}\mathbb{D}_{\infty }^{\alpha ,0;x}f\left( x\right) &=&\underset{N\rightarrow
\infty }{\lim }\left\{ ^{H}D_{N}^{\alpha ,0;x}f\left( x\right) \right\} \\
&=&\underset{N\rightarrow \infty }{\lim }\left\{ \left( -1\right) ^{n}\left( 
\frac{d}{dx}\right) ^{n}\text{ }I_{N}^{n-\alpha ;x}f\left( x\right) \right\}
\\
&=&\left( -1\right) ^{n}\frac{1}{\Gamma \left( n-\alpha \right) }\left( 
\frac{d}{dx}\right) ^{n}\underset{N\rightarrow \infty }{\lim }\left\{
\int_{x}^{N}\left( x-t\right) ^{n-\alpha -1}f\left( t\right) dt\right\} \\
&=&D_{-}^{\alpha }f\left( x\right),
\end{eqnarray*}
the Cassar fractional derivative.

\item If we consider $\psi \left( x\right) =x$ and taking the limit $\beta \rightarrow 1$ on both sides of Eq.(\ref{HIL}) and Eq.(\ref{HIL1}), we obtain
\begin{eqnarray*}
\frac{\text{ }^{H}\mathbb{D}_{a+}^{\alpha ,1;x}f\left( x\right) +\left(
-1\right) ^{n}\text{ }^{H}\mathbb{D}_{b-}^{\alpha ,1;x}f\left( x\right) }{%
2\cos \left( \frac{\pi \alpha }{2}\right) } &=&\frac{^{C}D_{a+}^{\alpha
}f\left( x\right) +\left( -1\right) ^{n}\text{ }^{C}D_{b-}^{\alpha }f\left(
x\right) }{2\cos \left( \frac{\pi \alpha }{2}\right) } \\
&=&\text{ }_{RC}D^{\alpha }f\left( x\right) ,
\end{eqnarray*}
the Caputo-Riesz fractional derivative.
\end{enumerate}

Then, from this wide class of integrals and fractional derivatives that are particular cases of the $\psi$-Hilfer fractional derivative, we conclude: our fractional derivative, introduced in section 2, really this a generalization of numerous fractional derivatives.

%%%%%%%%%%%%%%%%%%%%%%%%%%%%%%%%%%%%%%%%%%%%%%%%%%%%%%%%%%%%%%%%%%%%%%%%%%%%%%%%%%%%%%%%%%%%%%%%%%%%%%%%%%%%%%%%%%%%%%%%%%%%%%%%%%%%%%%%%%%%%%%%%%%%%%%%%%%%%%%%%%%%%%%%%%%%%%%%%%%%%%%%%%%%%%%%
\section{Concluding remarks}

The main objective of this paper was to propose a new fractional derivative with respect to another function $\psi$, in the sense of the Hilfer fractional derivative. We discussed some important and interesting results from the $\psi$-Hilfer fractional operator and a example involving Mittag-Leffler functions. We discussed the importance of the operator in order to overcome a wide number of definitions, presenting a class of integrals and fractional derivatives.

An interesting question is to present a generalization of Gronwall inequality via fractional integral of $f$ with respect to another function $\psi$ and to study the existence and uniqueness of Cauchy-type problem by means of the $\psi$-Hilfer fractional derivative. On the other hand, seems to be possible a generalization for the $\psi$-Hilfer fractional operator, just considering the variable order $\alpha(x)$ and type $0\leq\beta\leq 1 $. These topics are subjects of study in a forthcoming paper \cite{JOSE}.

%%%%%%%%%%%%%%%%%%%%%%%%%%%%%%%%%%%%%%%%%%%%%%%%%%%%%%%%%%%%%%%%%%%%%%%%%%%%%%%%%%%%%%%%%%%%%%%%%%%%%%%%%%%%%%%%%%%%%%%%%%%%%%%%%%%%%%%%%%%%%%%%%%%%%%%%%%%

\bibliography{ref}
\bibliographystyle{plain}

\end{document}